\documentclass[a4paper,11pt,reqno]{amsart}
\usepackage{amsmath,amssymb,amsfonts,amsthm}
\usepackage{latexsym,textcomp}
\usepackage{dsfont}
\usepackage[german, english]{babel}
\usepackage{hyperref}
\usepackage[capitalize,nameinlink]{cleveref}
\usepackage{enumerate}
\usepackage[utf8]{inputenc}
\usepackage{faktor}

\usepackage{color}

\newtheorem{lemma}{Lemma}[section]
\newtheorem{theorem}[lemma]{Theorem}

\newtheorem{proposition}[lemma]{Proposition}
\newtheorem{corollary}[lemma]{Corollary}

\theoremstyle{definition}
\newtheorem{definition}[lemma]{Definition}
\newtheorem{example}[lemma]{Example}
\newtheorem{remark}[lemma]{Remark}

\numberwithin{equation}{section}
\newcommand{\comment}[1]{}

\newcommand{\cR}{{\mathcal R}}
\newcommand{\cE}{{\mathcal E}}

\newcommand{\cF}{{\mathcal F}}

\newcommand{\R}{{\mathbb R}}

\newcommand{\N}{{\mathbb N}}

\newcommand{\vr}{{\varrho}}

\newcommand{\av}[1]{\left\Vert #1\right\Vert}

\newcommand{\Hm}[1]{\leavevmode{\marginpar{\tiny%
$\hbox to 0mm{\hspace*{-0.5mm}$\leftarrow$\hss}%
\vcenter{\vrule depth 0.1mm height 0.1mm width \the\marginparwidth}%
\hbox to 0mm{\hss$\rightarrow$\hspace*{-0.5mm}}$\\\relax\raggedright
#1}}}

\begin{document}

\title[Nonlinear Resistance forms]{Nonlinear resistance forms}

%

\author{Simon Puchert}
\address{S. Puchert, Mathematisches Institut \\Friedrich-Schiller-Universität Jena \\07743 Jena, Germany } \email{simon.puchert@uni-jena.de}

\author{Marcel Schmidt}
\address{M. Schmidt, Mathematisches Institut, Universität Leipzig, 04109 Leipzig, Germany.} \email{marcel.schmidt@math.uni-leipzig.de}

\maketitle

\begin{abstract}
In this paper we introduce the notion of nonlinear resistance forms. We define a $1$-parameter family of nonlinear resistance metrics and show their additivity over serial circuits. Moreover, we prove that resistance forms and $p$-resistance forms fall into our framework. 
\end{abstract}




\section{Introduction}

Resistance forms were introduced by Kigami \cite{Kig01,Kig03} as an abstract framework to study Laplacians and Brownian motion on fractals. For an excellent account on how they were conceived from previous attempts on defining Laplacians and Brownian motion on special classes of fractals, we refer to the introduction of \cite{KS25}.

Resistance forms  can be seen as a natural generalization (or a limit) of quadratic  energies on (infinite) discrete networks, whose connection to Markov chains and electrical networks is well-known, see e.g. \cite{DS84,Soa94} and references therein. Nonlinear versions of these discrete network energies were studied \cite{MS90,Soa93,Soa93a,Soa94,Kas16,Kas20,HKO21}, but it seems that the momentum of this direction of research has been somewhat limited  thus far. In contrast, recent years have seen a surge of publications on $p$-energies on fractals \cite{HPS04,CGQ22,BL23,Kig23,GYZ23,Bau24,CGYZ24,Shi24,AB25,MS25} (this list is not comprehensive), which are a $p$-homogeneous version of Kigami's resistance forms that lead to  $p$-Laplacians on fractals. Based on these works, in \cite{KS25} Kajino and Shimizu introduce the notion of $p$-resistance forms, which is aimed at providing an abstract framework for studying the aforementioned $p$-energies akin to the theory of Kigami's resistance forms. 

In this paper we introduce an even more general class of functionals, which we call {\em nonlinear resistance forms} (Definition~\ref{definition:nonlinear resistance form}), that encompasses all the functionals mentioned previously. Our aim is to single out  minimal assumptions on convex functionals (living on the space of all real-valued functions on a given set) such that they possess the following  three key features: 
\begin{enumerate}[(I)]
 \item Approximability by continuous convex functionals on function spaces over finite sets. 
 \item Compatibility with normal contractions. 
 \item The notion of a resistance metric that is additive over serial circuits. 
\end{enumerate}
  The tools that we use are mostly taken from abstract convex analysis. Once acquainted with them, they simplify several proofs and make them and some definitions more transparent, when compared to the literature on $p$-resistance forms. For $p$-resistance forms our approach also allows to treat the case $p = 1$, which is excluded in the literature. More generally, we can leave behind many remnants of the linear theory, such as e.g. homogeneity of the functionals. 
  
  Our nonlinear resistance forms come with two types of resistances, the {\em elementary resistance}, which is defined naively as in the linear case (Definition~\ref{definition:elementary resistance}),  and  the {\em $t$-resistance} (Definition~\ref{definition:resistance}), where $t > 0$ is a real parameter, which enjoys (III). The main insights that we prove in this text are the following: 
\begin{enumerate}[(A)]
 \item Property (I) is intimately linked to the lower semicontinuity of the functional, see Theorem~\ref{theorem:approximating forms} and Corollary~\ref{corollary:approximation via finite subsets}. Under some assumption on reflexivity and finiteness of the elementary resistance, lower semicontinuity (and hence approximability) can be characterized in terms of completeness of the form domain (or the modular space) with respect to the so-called {\em Luxemburg seminorm}, see Theorem~\ref{theorem:charcterization lower semicontinuity constant kernel} and Theorem~\ref{theorem:charcterization lower semicontinuity trivial kernel}.
 
 \item The triangle inequality for the family of $t$-resistances  is intimately linked to the compatibility with normal contractions (II), whose definition we took from \cite{Puc25}. We characterize this compatibility in Theorem~\ref{proposition:reduction contraction} and discuss  why it can be seen as the minimal compatibility with normal contractions that is required for proving  the triangle inequality for the $t$-resistance, see  Remark~\ref{remark:triangle inequality and contractions}.
 
 \item  In the $p$-homogeneous case, the $t$-resistance is just a power of the elementary resistance, see Theorem~\ref{theorem:homogeneous resistance}.   In particular, in the case of Kigami's resistance forms, our $t$-resistance coincides with his notion of resistance. 
 
 \item  The $t$-resistance is additive over serial circuits, where two ways of defining serial circuits in terms of the nonlinear resistance forms are considered, see Theorem~\ref{theorem:additivity serial circuits} and Theorem~\ref{theorem:additivity serial circuitsII}. Together with (B), this means that our $t$-resistance  satisfies (III) and besides (C), it is another justification for our definition. 
 
 \item We show that the previously considered nonlinear functionals on discrete spaces and fractals fall into our framework, see Section~\ref{section:examples}. In the case of $p$-resistance forms, this is actually quite nontrivial and relies on the results discussed in (A). 
\end{enumerate}

Our paper is organized as follows: In  Section~\ref{section:convex analysis} we review basics on convex analysis and characterize lower semicontinuity of convex functionals on locally convex topological vector spaces. In Section~\ref{section:elementary resistance} these abstract results are applied to convex functionals on $\cF(X)$, the space of all real-valued functions on $X$ equipped with the topology of pointwise convergence. In this context the elementary resistance is introduced. Section~\ref{section:nonlinear resistance forms} is the main part of our paper. It introduces nonlinear resistance forms and corresponding resistance metrics and discusses their properties. In Section~\ref{section:examples} we give several examples and show that Kigami's resistance forms and the $p$-resistance forms of Kajino and Shimizu fall into our framework.

%

\section{Some results in convex analysis}  \label{section:convex analysis}

\subsection{Convex functionals on locally convex spaces} \label{subsection:convex functionals}
In this subsection we recall some basic notions from convex analysis. For more background, see e.g. \cite{ET99,Zua02}. We are certain that  the results presented here are known but could not find references for all of them in the needed generality. For this reason and for the convenience of the reader we provide (most) proofs.

Let $V$ be an $\R$-vector space and let $f \colon V \to (-\infty,\infty]$. Its {\em effective domain} is $D(f) = \{x \in V \mid f(x) < \infty\}$ and we say that $f$ is {\em proper} if $D(f) \neq \emptyset$. 

In this subsection we assume that $\vr \colon V \to [0,\infty]$ is convex and satisfies $\vr(0) = 0$. In particular, it is proper. 

We say that $\vr$ is {\em symmetric} if $\vr(x) = \vr(-x)$ for all $x \in V$. For arbitrary  $\vr$ we define its {\em symmetrization} $\vr_s \colon V \to (-\infty,\infty]$ by
$$\vr_s(x) = \frac{1}{2} \left(\vr(x) + \vr(-x)\right).$$
 Since $\vr(0) = 0$, the functional $\vr_s$ is also proper. Moreover,  $D(\vr_s) = \{x \in D(\vr) \mid  -x \in D(\vr)\}$ and  $\vr_s = \vr$ if and only if $\vr$ is symmetric.  


The {\em modular cone} of $\vr$ is defined by
$$M(\vr) = \{x \in V \mid \lim_{\alpha \to 0+} \vr(\alpha x) =  0\}. $$
 The convexity of $\vr$ and $\vr(0) = 0$ imply $M(\vr) = \{\lambda x \mid \lambda \geq 0, x\in D(\vr)\}$.   On $M(\vr)$ the set $\{x \in V \mid \vr(x) \leq 1\}$ is absorbing. We denote the induced Minkowski functional by
$$\av{\cdot}_L = \av{\cdot}_{L,\vr} \colon M(\vr) \to [0,\infty),\quad \av{x}_L = \inf \{ \lambda > 0 \mid \vr(\lambda^{-1}x) \leq 1 \}$$
and call it the {\em Luxemburg functional} of $\vr$. By general theory it is sublinear (i.e. $\av{\lambda x}_L = \lambda \av{x}_L$ and $\av{x+y}_L \leq \av{x}_L + \av{y}_L$ for all $\lambda \geq 0$ and $x,y \in M(\vr)$) and the convexity of $\vr$ implies $\av{x}_L \leq 1 + \vr(x)$, $x \in M(\vr)$. If $\vr$ is symmetric, then $M(\vr)$ is a vector space and $\av{\cdot}_L$ is a seminorm, the so-called {\em Luxemburg seminorm} of $\vr$. For later purposes it is convenient to let $\av{x}_L = \infty$ for $x \in V \setminus M(\vr)$ such that $M(\vr)$ can be seen as the effective domain of $\av{\cdot}_L$. 

The functional $\vr$ is called {\em left-continuous} if 
$$\lim_{\lambda \to 1-} \vr(\lambda x) = \vr(x)$$
for all $x \in V$. In this case, it follows directly from the definition that
$$\{x \in V \mid \vr(x) \leq 1\} = \{x \in V \mid \av{x}_L \leq 1\}.$$
\begin{remark}
For $x \in V$ the inequality  $\limsup_{\lambda \to 1-} \vr(\lambda x) \leq \vr(x)$ is a consequence of the convexity of $\vr$ and $\vr(0) = 0$. Hence, for the functionals considered here, left-continuity is equivalent to {\em left-lower semicontinuity}, i.e.,  $\vr(x) \leq \liminf_{\lambda \to 1-} \vr(\lambda x)$ for all $x \in V$. Below we will consider lower semicontinuous $\vr$ with respect to vector space topologies on $V$, where the latter property is automatically satisfied. In this case, we will quite frequently use the identity of the unit balls of $\vr$ and $\av{\cdot}_L$.  
\end{remark}

Next we assume that we are given a Hausdorff locally convex topology $\mathfrak{T}$ on $V$. We say that $f \colon V \to (-\infty,\infty]$ is {\em lower semicontinuous} if its epigraph 
$${\rm epi}\,(f) =  \{(x,t) \in V \times \R \mid f(x) \leq t\}$$
is closed in $V \times \R$ (which is equipped with the product topology).  The standard characterization of closed sets via nets yields that the lower semicontinuity of $f$ is equivalent to the following: For each net $(x_i)$ in $V$ and $x \in V$ with $\lim x_i = x$ with respect to $\mathfrak{T}$ we have 
$$f(x) \leq \liminf f(x_i),$$
where $\liminf f(x_i) = \sup_i \inf_{i \prec j} f(x_j)$ and $\prec$ denotes the preorder of the index set of the net. 

If for every sequence $(x_n)$ with $\lim_{n \to \infty} x_n = x$ with respect to $\mathfrak{T}$ we have 
$$f(x) \leq \liminf_{n \to \infty} f(x_n),$$
then we call $f$ {\em sequentially lower semicontinuous}. Clearly, lower semicontinuity implies sequential lower semicontinuity. The converse holds e.g. if $\mathfrak{T}$ is metrizable but also in more general settings, see the discussion below.  

We denote the dual space of $(V,\mathfrak{T})$ by $V' = (V,\mathfrak{T})'$ and write 
$$(\cdot,\cdot) \colon V' \times V \to \R,\quad (\varphi,x) = \varphi(x)$$
for the dual pairing between $V'$ and $V$. Since $(V,\mathfrak{T})$ is Hausdorff and locally convex,  the Hahn-Banach theorem implies that $V'$ separates the points of $V$. Hence, we have  $(V',\sigma(V',V))' = V$, where $V$ is identified with $\{(\cdot,x) \mid x \in V\}$ and $\sigma(V',V)$ denotes the weak-*-topology on $V'$.

We define the {\em convex conjugate functional} $\vr^*$ of $\vr$  with respect to $\mathfrak{T}$ by
$$\vr^* \colon V' \to [0,\infty], \quad \vr^*(\varphi) = \sup \{(\varphi,x) - \vr(x) \mid x \in V\}.$$
Since $\vr$ is proper, the supremum in the definition can actually be taken over $x \in D(\vr)$. Moreover, $\vr \geq 0$ and $\vr(0) = 0$ imply $\vr^* \geq 0$ and $\vr^*(0) = 0$. As a supremum of $\sigma(V',V)$-continuous convex functionals, $\vr^*$ is a lower semicontinuous convex functional on $(V',\sigma(V',V))$.  If  $\vr$ is lower semicontinuous with respect to $\mathfrak{T}$, then $\vr^{**} = \vr$, where $V$ is identified with $(V',\sigma(V',V))'$. More precisely, for all $x \in  V$ we have
$$ \vr(x) = \sup \{(\varphi,x) - \vr^*(\varphi) \mid \varphi \in V'\}.$$
This is known as Fenchel-Moreau theorem, see e.g. \cite[Theorem 2.3.3]{Zua02}.  Because of this observation, in this text $\vr^{**}$ is always considered to be defined with respect to $\sigma(V',V)$ on $V'$.

The following elementary example is the prototype for the duality - it will be used below.

\begin{example}\label{example:basic duality}
 Let $V = \R$ such that $V' = \R$  (via the dual pairing $(s,t) = st$, $s,t \in \R$). For $1 \leq p < \infty$ consider the function $f_p\colon \R \to \R$, $f_p(t) = p^{-1} |t|^p$.  If $1 < p < \infty$ and $q^{-1} + p^{-1} = 1$, then $(f_p)^* = f_q$ and for $p = 1$ we have
 $$(f_1)^* \colon \R \to \{0,\infty\}, \quad (f_1)^*(t) = \begin{cases}
                 0 &\text{if }|t| \leq 1\\
                 \infty &\text{if } |t| > 1
                \end{cases}.
$$
\end{example}

Next we study a functional that is in some form of duality with the Luxemburg functional. More precisely, we consider the {\em Orlicz functional}  of $\vr$ defined by 
$$\av{\cdot}_O = \av{\cdot}_{O,\vr} \colon V \to [0,\infty],\quad \av{x}_O = \sup\{(\varphi,x) \mid \varphi \in V',  \vr^*(\varphi) \leq 1\}.$$

\begin{remark}
 \begin{enumerate}[(a)]
  \item The Luxemburg functional only depends on $\vr$, whereas $\vr^*$ and hence the Orlicz functional also depend on the choice of the locally convex topology on $V$.
  \item The names Luxemburg functional and Orlicz functional (respectively seminorm) are borrowed from the theory of Orlicz spaces and their generalizations (Musielak–Orlicz space etc.), see e.g. \cite{DHHR11}.  
 \end{enumerate}

\end{remark}
If $\vr$ is lower semicontinuous, then $\vr^{**} = \vr$  implies that the Orlicz functional of $\vr^*$ is given by
$$\av{\varphi}_{O,\vr^*} = \sup\{(\varphi,x) \mid x \in V, \vr(x) \leq 1\}.$$

The next proposition shows that $\av{\cdot}_O$ is finite on $M(\vr)$. For lower semicontinuous $\vr$ also the converse holds, i.e., $\av{x}_O < \infty$   implies $x \in M(\vr)$.   As for the Luxemburg functional, the Orlicz functional is sublinear on $M(\vr)$ and, if $\vr$ is symmetric, it is even a seminorm on $M(\vr)$  -  the so-called {\em Orlicz seminorm}.

\begin{proposition}[Fundamental inequalities]\label{proposition:fundamental inequalities}

 \begin{enumerate}[(a)]
  
  \item For all $x \in M(\vr)$ we have $\av{x}_O \leq 2\av{x}_L$. In particular, $\av{\cdot}_O$ is finite on $M(\vr)$. 
  
  \item If $\varphi \in M(\vr^*)$ and $x \in V$ with $\av{x}_{O,\vr} < \infty$, then 
  $$(\varphi,x) \leq \av{\varphi}_{L,\vr^*} \av{x}_{O,\vr}.$$
  \end{enumerate}
  If, additionally, $\vr$ is lower semicontinuous, then also the following hold.
  \begin{enumerate}[(a)] \setcounter{enumi}{2}
   
   \item Let $x \in V$. Then $\av{x}_O < \infty$ if and only if $x \in M(\vr)$. Moreover, for all $x \in M(\vr)$ we have $$\av{x}_L \leq \av{x}_O \leq 2\av{x}_L.$$
  
  \item  If $\varphi \in V'$ with $\av{\varphi}_{O,\vr^*} < \infty$ and $x \in M(\vr)$, then 
  $$(\varphi,x) \leq \av{\varphi}_{O,\vr^*} \av{x}_{L,\vr}.$$
 \end{enumerate}

\end{proposition}
 \begin{proof}
 (a):  Let $x \in M(\vr)$ and let $\varphi \in V'$ with $\vr^*(\varphi) \leq 1$. The definition of $\vr^*$ implies $(\varphi,y) - \vr(y) \leq 1$ for all $y \in V$. For $\lambda > \av{x}_L$ and $y = \lambda^{-1}x$ this inequality implies
  $$\lambda^{-1} (\varphi,x) = (\varphi,y) \leq 1 + \vr(y) \leq 2,$$
  where we used $\vr(y) = \vr(\lambda^{-1}x) \leq 1$ for the last inequality. Taking the supremum over all $\vr^*(\varphi) \leq 1$ and letting $\lambda \searrow \av{x}_L$ yields the claim. In particular, this shows the finiteness of $\av{\cdot}_O$ on $M(\vr)$. 
 
 (b):  Let $\lambda > \av{\varphi}_{L,\vr^*}$ and, therefore, $\vr^*(\lambda^{-1} \varphi) \leq 1$. We infer 
 $$(\varphi,x) = \lambda (\lambda^{-1}\varphi,x) \leq \lambda \av{x}_{O,\vr}.$$
 Letting $\lambda \searrow \av{\varphi}_{L,\vr^*}$ yields the claim.

 (c): The 'if part' of the statement and the second inequality were already proven in (a). 
 
 Let $x\in V$ and let $ \av{x}_{O,\vr} < \lambda < \infty$. Using $\vr = \vr^{**}$, (b) and $\av{\cdot}_{L,\vr^*} \leq 1 + \vr^*$ on $D(\vr^*)$, we infer
  \begin{align*}
   \vr(\lambda^{-1} x) &= \sup \{(\varphi,\lambda^{-1} x) - \vr^*(\varphi) \mid \varphi \in D(\vr^*)\}\\
   &\leq \sup \{\lambda^{-1} \av{\varphi}_{L,\vr^*} \av{x}_{O,\vr} - \vr^*(\varphi) \mid \varphi \in D(\vr^*)\} \\
   &\leq \sup \{ \av{\varphi}_{L,\vr^*}   - \vr^*(\varphi) \mid \varphi \in D(\vr^*)\}  \leq 1.
  \end{align*}
  This shows $x \in M(\vr)$, $\lambda \geq \av{x}_{L,\vr}$ and yields the desired inequality after letting $\lambda \searrow \av{x}_{O,\vr}$.

  (d):  This follows directly from (b) applied to $\vr^*$ and using $\vr = \vr^{**}$.
 \end{proof}

  If $\vr$ is lower semicontinuous and symmetric, then (c) of the previous proposition shows that the Luxemburg seminorm and the Orlicz seminorm are equivalent seminorms on $M(\vr)$. Hence, as sets the dual spaces  of $(M(\vr),\av{\cdot}_L)'$ and $(M(\vr),\av{\cdot}_O)'$ coincide and we simply write $M(\vr)'$ for them. However, they are equipped with different equivalent operator norms. We denote by $\av{\cdot}_{L'}$ the operator norm on $(M(\vr),\av{\cdot}_L)'$ and by $\av{\cdot}_{O'}$ the operator norm on $(M(\vr),\av{\cdot}_O)'$.

  \begin{corollary}[Duality of Luxemburg and Orlicz seminorms] \label{coro:dual spaces}
   Assume that $\vr$ is lower semicontinuous and symmetric. Then $M(\vr^*) \subset M(\vr)'$ and 
   $$\av{\varphi}_{O,\vr*} = \av{\varphi}_{L'}  \text{ and }\av{\varphi}_{L,\vr*} = \av{\varphi}_{O'} $$
   for all $\varphi \in M(\vr^*)$.
  \end{corollary}
 \begin{proof}
  Using $\vr^{**} = \vr$ and the symmetry of $\vr$ and $\vr^*$, we obtain 
  $$\av{\varphi}_{O,\vr*} = \sup \{|(\varphi,x)| \mid x \in V, \vr(x) \leq 1\} = \sup \{|(\varphi,x)| \mid x \in V, \av{x}_L \leq 1\},$$
  where we used $\vr(x) \leq 1$ if and only if $\av{\cdot}_L \leq 1$ for the second equality. This shows $M(\vr^*) \subset M(\vr)'$ and the first identity for the operator norms. 
  
 As for the second equality, the previous proposition yields $\av{\varphi}_{O'} \leq \av{\varphi}_{L,\vr^*}$ and  it remains to prove the opposite inequality.  For $\lambda  > \av{\varphi}_{O'}$ we obtain  
 \begin{align*}
\vr^*(\lambda^{-1} \varphi) &= \sup \{ (\lambda^{-1} \varphi,y) - \vr(y) \mid y \in D(\vr)\}   \\
&\leq \sup \{\av{y}_{O,\vr}  - \vr(y) \mid y \in D(\vr)\}.
 \end{align*}
 Hence, it suffices to show that the right side of this inequality is less or equal than $1$.  For $y \in D(\vr)$ and $\varepsilon > 0$ we choose $\psi \in V'$ with $\vr^*(\psi) \leq 1$ and $\av{y}_O \leq (\psi,y) + \varepsilon$. Since by definition $\vr(y) \geq (\psi,y) - \vr^*(\psi) \geq (\psi,y) - 1$, we conclude 
 $$\av{y}_{O,\vr} - \vr(y) \leq (\psi,y) + \varepsilon - (\psi,y) + 1 \leq 1 + \varepsilon. $$
 Since $\varepsilon > 0$ was arbitrary, this yields the claim.   
 \end{proof}

 Under some assumptions on $\vr$, its effective domain is a cone. In this case, we have the useful identity $D(\vr) = M(\vr)$.   One such condition is discussed next.  The functional $\vr$ is said to satisfy the {\em $\Delta_2$-condition} if there exists $C \geq 0$ such that $\vr(2f) \leq C \vr(f)$ for all $f \in D(\vr)$. Then, by convexity $D(\vr)$ is a positive cone and we obtain $D(\vr) = M(\vr)$. 
 
 The $\Delta_2$-condition can be characterized with the help of the convex conjugate functional. To this end, we introduce the {\em $\nabla_2$-condition}, which is satisfied by $\vr$ if there exists $K > 2$ such that $K \vr(f) \leq \vr(2f)$ for all $f \in D(\vr)$. 
 
\begin{remark}
 Due to the convexity of $\vr$ and $\vr(0) = 0$, the constant $C$ in the $\Delta_2$-condition must satisfy $C \geq 2$ unless $\vr(x) =  0$ for all $x \in D(\vr)$. Moreover, by convexity and $\vr(0) = 0$, we have $K\vr(x) \leq \vr(2x)$ for all $0 \leq K \leq 2$. Hence, in the $\nabla_2$-condition, the assumption $K  > 2$ is essential.    
\end{remark}

 \begin{lemma}[Duality of the $\Delta_2$-condition and the $\nabla_2$-condition]\label{lemma:duality delta and nabla}
 \begin{enumerate}[(a)]
  \item   If $\vr$ satisfies the $\Delta_2$-condition, then $\vr^*$ satisfies the $\nabla_2$-condition.  
  \item If $\vr$ satisfies the $\nabla_2$-condition, then $\vr^*$ satisfies the $\Delta_2$-condition. 
  \item Assume that $\vr$ is lower semicontinuous. Then  $\vr$ satisfies the $\Delta_2$-condition if and only if $\vr^*$ satisfies the $\nabla_2$-condition and $\vr$ satisfies the $\nabla_2$-condition if and only if $\vr^*$ satisfies the $\Delta_2$-condition. 
 \end{enumerate}

 \end{lemma}
\begin{proof}
 (a): Assume that $\vr$ satisfies the $\Delta_2$-condition, i.e., $\vr(2x) \leq C \vr(x)$ for all $f \in D(\vr)$. For $\varphi \in V'$ we obtain 
 \begin{align*}
 \vr^*((C/2)\varphi) &= \sup \{ ((C/2)\varphi, x) - \vr(x) \mid x \in V\}\\
 &= \sup \{ C(\varphi, x) - \vr(2x) \mid x \in V\}  \\
  &\geq C \vr^*(\varphi).
 \end{align*}
We can assume $C  > 2$  (else replace $C$ by a larger constant).  Then $\varphi = \lambda (2/C) \varphi  + (1-\lambda)2\varphi$ with $\lambda = C /(2C -2)$. Using this, $0 \leq \lambda \leq 1$ and $C \vr^*((2/C)\varphi) \leq \vr^*(\varphi)$, we obtain 
 $$\vr^*(\varphi) \leq \lambda \vr^*((2/C)\varphi) + (1-\lambda) \vr^*(2\varphi) \leq \lambda/C \vr^*(\varphi) + (1-\lambda)\vr^*(2 \varphi).$$
 Now assume that $2\varphi \in D(\vr^*)$ (else there is nothing to show). Using the convexity of $\vr^*$ and $\vr^*(0) = 0$,  we infer $\varphi \in D(\vr^*)$ and obtain the inequality 
 $$\frac{(1 - \lambda/C)}{1-\lambda} \vr^*(\varphi) \leq \vr^*(2\varphi).$$
 Since $(1 - \lambda/C)/(1-\lambda)  = (2C - 3)/(C-2) > 2$, we obtain (a).

 (b): Assume that $\vr$ satisfies the $\nabla_2$-condition, i.e., there exists $K  > 2$ with $\vr(2x) \geq K \vr(x)$ for all $x \in V$.   The same computation as in (a) yields $\vr^*((K/2) \varphi) \leq K \vr^*(\varphi)$. Since $K > 2$, we find $n \in \N$ such that $K^n /2^n \geq 2$. Iterating the previous inequality and using the convexity of $\vr^*$, we obtain 
 $$\vr^*(2 \varphi) \leq \vr^*(K^n/2^n \varphi) \leq K^n \vr^*(\varphi). $$

 (c): This follows from (a) and (b) using $\vr^{**} = \vr$.
\end{proof}

  In general it is hard to compute the Luxemburg functional, the Orlicz functional and the convex conjugate of $\vr$. However, for homogeneous $\vr$ this is possible  and will be discussed next.   Let $1 \leq p <  \infty$. We say that $\vr$ is {\em positively $p$-homogeneous} if $\vr(\lambda x) = \lambda^p \vr(x)$ for all $x \in V$ and $\lambda \geq 0$.   Clearly, positively homogeneous functionals satisfy the $\Delta_2$-condition and, if $p > 1$, also the $\nabla_2$-condition.


\begin{proposition}[The homogeneous case] \label{prop:seminorms for homogenous functionals}
Assume that for some $1 \leq p < \infty$ the functional $\vr$ is positively $p$-homogeneous and let $1 < q \leq \infty$ with $q^{-1} + p^{-1} = 1$. Then, $M(\vr) = D(\vr)$ and the following hold: 
\begin{enumerate}[(a)]
 \item For all $x \in M(\vr)$ we have $\av{x}_L = \vr(x)^{1/p}$. 
 \item If $p > 1$, then $\vr^*$ is positively $q$-homogeneous. If $p = 1$, then 
 $$\vr^*(\varphi) = \begin{cases}
                    0 &\text{if } \varphi \leq \av{\cdot}_L\\
                    \infty& \text{else } 
                   \end{cases}.
$$
 \item If $\vr$ is lower semicontinuous, then for all $x \in D(\vr)$ we have
 $$\av{x}_L^p = \vr(x) = \begin{cases}
                          \frac{q-1}{q^p} \av{x}_{O,\vr}^p &\text{if } p > 1\\
                          \av{x}_{O,\vr} &\text{if } p = 1
                         \end{cases}.$$
\end{enumerate}
\end{proposition}

\begin{proof}
(a): This is trivial. 

(b): $p > 1$: Let $\varphi \in D(\vr^*)$ and $\lambda \geq 0$. Using $\vr(\lambda^{q-1} x) = \lambda^{p(q-1)} \vr(x) = \lambda^q \vr(x)$, we obtain 
 \begin{align*}
  \lambda^q \vr^*(\varphi) &= \sup \{(\lambda \varphi, \lambda^{q-1} x) - \lambda^q \vr(x) \mid x \in V\}\\
  &= \sup \{(\lambda \varphi, \lambda^{q-1} x) - \vr(\lambda^{q-1}x) \mid x \in V\}\\
  &= \vr^*(\lambda \varphi).
 \end{align*}

 $p = 1$: Assume $\varphi \leq \av{\cdot}_L$. Since $\vr = \av{\cdot}_L$, the  assumed inequality implies $(\varphi,x) - \vr(x) \leq 0$ for all $x \in D(\vr)$. Since the left side of this inequality equals $0$ for $x = 0$, we infer $\vr^*(\varphi) = 0$. 
 
 Now assume there exists $x \in V$ with $(\varphi,x) > \av{x}_L = \vr(x)$. In particular, this implies $x \in M(\vr) = D(\vr)$. For $\lambda \geq 0$ we obtain 
 $$(\varphi,\lambda x) - \vr(\lambda x) = \lambda ((\varphi,x) - \av{x}_L) \to \infty, \quad \text{as } \lambda \to \infty, $$
 showing $\vr^*(\varphi) = \infty$. 
 
(c): $p > 1$: Case 1: $\vr(x)^{1/p} = \av{x}_L = 0$. By Lemma~\ref{proposition:fundamental inequalities} this is equivalent to $\av{x}_O = 0$. Hence, the desired equality holds true with both sides equal to $0$. 

Case 2:  $x \in D(\vr)$ with $\av{x}_L > 0$. Lemma~\ref{proposition:fundamental inequalities} shows $\av{x}_O > 0$. In particular,  there exists $\varphi \in V'$ with $(\varphi,x) > 0$ and $\vr^*(\varphi) \leq 1$. By the homogeneity of $\vr^*$ this implies $\vr^*(\varphi) > 0$, as otherwise we would have $\vr^*(\lambda \varphi) = \lambda^q \vr^*(\varphi) = 0$ for all $\lambda > 0$, which leads to $\av{x}_O  \geq \lambda (\varphi,x)$ for all $\lambda > 0$, a contradiction. 

With  this observation, a simple rescaling argument using the $q$-homogeneity of $\vr^*$ yields that for all $C > 0$ we have
\begin{align*}
 \av{x}_O &= \sup \{(\varphi,x) \mid \varphi \in V', 0 <  \vr^*(\varphi) \leq 1 \}\\
 &= C^{-\frac{1}{q}} \sup \{(\varphi,x) \mid \varphi \in V', \vr^*(\varphi) = C\}
\end{align*}
 Using this identity, $\vr = \vr^{**}$ and the duality discussed in Example~\ref{example:basic duality},  we obtain 
 \begin{align*}
  \vr(x) &= \sup\{(\varphi,x) - \vr^*(\varphi) \mid \varphi \in D(\vr^*)\} \\
  &= \sup_{C \geq 0} \sup \{ (\varphi,x) - C \mid \varphi \in V', \vr^*(\varphi) = C\}\\
  &=  \sup_{C \geq 0}  \left(C^\frac{1}{q} \av{x}_O - C \right) = q \sup_{C \geq 0}  \left(C \frac{\av{x}_O}{q} - \frac{C^q}{q} \right) \\
  &= q \sup_{C \in \R}  \left(C \frac{\av{x}_O}{q} - \frac{|C|^q}{q} \right)  = q \frac{1}{p} \left(\frac{\av{x}_O}{q}\right)^p = \frac{q-1}{q^p} \av{x}_O^p.
 \end{align*}
For the last equality we used $q/p = q - 1$. 

 $p = 1$: The inequality $\av{\cdot}_L \leq \av{\cdot}_O$ holds true in general, see Proposition~\ref{proposition:fundamental inequalities}. According to (b) we have  $\vr^*(\varphi) \leq 1$ if and only if $\varphi \leq \av{\cdot}_L$. Using this observation and  Proposition~\ref{proposition:fundamental inequalities},  for  $x \in D(\vr)$ we estimate
 \begin{align*}
 \av{x}_O &= \sup\{(\varphi,x) \mid \varphi \in V', \varphi \leq \av{\cdot}_L\} \\
 &\leq \av{x}_L \sup\{ \av{\varphi}_{O,\vr^*} \mid \varphi \in V', \varphi \leq \av{\cdot}_L\}.
 \end{align*}
 Using $\vr^{**} = \vr$, $\vr = \av{\cdot}_L$ and $\varphi \leq \av{\cdot}_L$, we infer 
 \begin{align*}
 \av{\varphi}_{O,\vr^*} &= \sup \{(\varphi,x) \mid x \in V, \vr(x) \leq 1\} \leq 1.¸
 \end{align*}
Combining these inequalities yields $\av{\cdot}_O \leq \av{\cdot}_L$. 
%
%
%
%
%
%
%
\end{proof}


%
%

For later purposes we state one last elementary lemma on convex functions.

\begin{lemma}\label{lemma:convexity estimate}
Let $f \colon V \to \R$ be convex. Then for each $x,y \in V$ and $\lambda \in \R$ with $| \lambda | \leq 1$ we have 
 $$f(x + \lambda y) + f(x - \lambda y) \leq f(x + y) + f(x - y).$$
\end{lemma}
\begin{proof}
Since $|\lambda| \leq 1$, the following are convex combinations: 
$$x \pm \lambda y = \frac{1 \pm \lambda}{2} (x + y) + \frac{1 \mp \lambda}{2} (x - y). $$
Hence, the statement follows from the convexity of $f$. 
\end{proof}

\subsection{Lower semicontinuity}

In this subsection we study lower semicontinuity  for symmetric functionals. More precisely, we assume that $\vr$ is as in Subsection~\ref{subsection:convex functionals} and,  additionally, we assume the symmetry of $\vr$ such that $\av{\cdot}_L$ is a seminorm. Unfortunately, the locally convex space in our application -  the space of all functions with the topology of pointwise convergence - need  not be metrizable  and so we have to deal with non-metrizable locally convex spaces.

As for normed spaces, the seminormed space $(M(\vr),\av{\cdot}_L)$ isometrically embeds into its bidual via the natural embedding $x \mapsto (\cdot,x)_{M(\vr)}$, where  $(\cdot,\cdot)_{M(\vr)}$ denotes the dual pairing between $M(\vr)$ and $M(\vr)'$. Note that this isometric embedding is not injective if $\av{\cdot}_L$ is not a norm. We say that $\vr$ is {\em reflexive} if the image of $(M(\vr),\av{\cdot}_L)$ under this natural embedding is dense in its bidual.  For more context on this notion, see \cite[Appendix~A]{SZ25}.

The locally convex topology $\mathfrak{T}$ on $V$ is generated by a family of seminorms $p_i$, $i \in I$. We denote by $\mathfrak{T}_\vr$ the locally convex topology on $M(\vr)$ generated by the family of seminorms $\av{\cdot}_L$ and $p_i|_{M(\vr)}$, $i \in I$. In particular, a net $(x_i)$ in $M(\vr)$ converges to $x \in M(\vr)$ with respect to $\mathfrak{T}_\vr$ if and only if $x_i \to x$ with respect to $\mathfrak{T}$ and $\lim \av{x_i - x}_L = 0$. 

Recall that a locally convex topological vector space $(V,\mathfrak{T})$ is called {\em complete} if every Cauchy net converges. Here, $(x_i)$ is called a {\em Cauchy net} if for every zero neighborhood $U$ there exists $i_U$ such that $x_i - x_j \in U$ for all $i,j \geq i_U$. If the topology is metrizable, $(V,\mathfrak{T})$ is complete if and only if any Cauchy sequence with respect to one/any translation invariant metric inducing $\mathfrak{T}$ converges.

For metrizable $\mathfrak{T}$ the following results on lower semicontinuity are contained in \cite[Appendix~A]{SZ25}, with the latter being based on \cite{Schmi3}, which only deals with quadratic forms. The  proofs given in \cite{SZ25}  extend to the locally convex case with some modifications. We give details for the convenience of the reader.

\begin{lemma}\label{lemma:weak convergence} 
 Assume that $\vr$ is  symmetric and reflexive and that either $\av{\cdot}_L$ is  $\mathfrak{T}$-lower semicontinuous (considered as a functional $V \to [0,\infty]$) or that $(M(\vr),\mathfrak{T}_\vr)$ is complete. If $(x_i)$ is a bounded net in $(M(\vr),\av{\cdot}_L)$ and $\mathfrak{T}$-converges to $x \in V$ with respect to $\mathfrak{T}$, then $x_i \to x$ weakly in $(M(\vr),\av{\cdot}_L)$. 
\end{lemma}
\begin{proof}
We have to show that every subnet of $(x_i)$ has a subnet weakly converging to $x$. Since every subnet of $(x_i)$ is $\av{\cdot}_L$-bounded and $\mathfrak{T}$-converges to $x$, it suffices to show that $(x_i)$ itself has a subnet weakly converging to $x$.

%
%
%
%

The reflexivity of $\vr$ implies the weak compactness of balls in the completion of $(M(\vr),\av{\cdot}_L)$. Hence, we obtain a subnet $(y_j)$ of $(x_i)$ that converges weakly to some $\overline x$ in the completion of $(M(\vr),\av{\cdot}_L)$. Moreover, as a subnet $(y_j)$ still $\mathfrak{T}$-converges to $x$. Since the kernel of every functional in $M(\vr)'$ (extended to the completion) contains $\ker \av{\cdot}_L$, it suffices to show $\av{\overline x - x}_L = 0$ to deduce $y_j \to x$ weakly in $(M(\vr),\av{\cdot}_L)$. 

Since weak and strong closures of convex sets coincide in normed spaces, we find $(z_k)$ with the following properties:
 \begin{itemize}
  \item $(z_k)$  converges to $\overline{x}$ in the completion of $(M(\vr),\av{\cdot}_L)$ with respect to $\av{\cdot}_L$. In particular, $(z_k)$ is $\av{\cdot}_L$-Cauchy. 
  \item $(z_k)$ is a finite convex combination of the elements of $\{y_{j} \mid j \succ k\}$.
  \end{itemize}
  Since $\mathfrak{T}$ is locally convex, these convex combinations of the net $(y_j)$ with $\mathfrak{T}$-limit $x$ also $\mathfrak{T}$-converge to $x$.
  
  Case 1: $\av{\cdot}_L$ is lower semicontinuous: The lower semicontinuity and the $\av{\cdot}_L$-Cauchyness of $(z_k)$ imply 
\begin{align*}
\av{x - \overline{x}}_L &= \lim_{k} \av{x - z_k}_L \leq \liminf_{k,l} \av{z_l  - z_k}_L = 0.  
\end{align*}

Case 2:  $(M(\vr),\mathfrak{T}_\vr)$ is complete: The $\mathfrak{T}$ convergence and the $\av{\cdot}_L$-Cauchyness of $(z_k)$ imply that $(z_k)$ is $\mathfrak{T}_\vr$-Cauchy. Hence, by completeness, it $\mathfrak{T}_\vr$-converges to some $\hat x \in M(\vr)$. Since it already $\mathfrak{T}$-converges to $x$ and $\mathfrak{T}_\vr$-convergence is stronger than $\mathfrak{T}$ convergence and $\mathfrak{T}$ is Hausdorff, we infer $x =  \hat x$. But since $\mathfrak{T}_\vr$-convergence also implies $\av{\cdot}_L$-convergence, we obtain 
\begin{align*}
 0 &= \lim_k \av{z_k - x}_L = \av{\overline x - x}_L. \hfill \qedhere 
\end{align*}
\end{proof}

\begin{lemma}\label{lemma:equivalence of lower semicontinuity}
 Let $\vr$ be symmetric and reflexive. Then the following assertions are equivalent. 
 \begin{enumerate}[(i)]
  \item $\vr$ is lower semicontinuous. 
  \item $\vr$ is left-continuous and $\av{\cdot}_L$ is lower semicontinuous (as a map $V \to [0,\infty]$). 
 \end{enumerate}
\end{lemma}
\begin{proof}
(i) $\Rightarrow$ (ii): As discussed above, lower semicontinuity implies left-continuity and so $\{x \in V \mid \vr(x) \leq 1\} = \{x \in V \mid \av{x}_L \leq 1\}$. Since by a simple scaling argument $\av{\cdot}_L$ is $\mathfrak{T}$-lower semicontinuous if and only if the latter set is $\mathfrak{T}$-closed, we obtain the lower semicontinuity of $\av{\cdot}_L$.

(ii) $\Rightarrow$ (i):  Let $(x_i)$ be a net in $V$ such that $x_i \to x$ with respect to $\mathfrak{T}$. We can assume $\liminf \vr(x_i)  < \infty$ (else there is nothing to show)  and hence even $R = \sup_i \vr(x_i) < \infty$ (else pass to a suitable subnet). By convexity and $\vr(0) = 0$, we obtain $\av{x_i}_L \leq 1 + R $ for all $i$. Hence, $(x_i)$ is $\av{\cdot}_L$-bounded and we infer $x_i \to x$ weakly in $(M(\vr),\av{\cdot}_L)$ from Lemma~\ref{lemma:weak convergence}. Since $\vr$ is lower semicontinuous when considered as a functional on  $(M(\vr),\av{\cdot}_L)$ (see \cite[Corollary A.4]{SZ25}, which uses left-continuity) and on normed spaces  lower semicontinuous convex functionals are  lower semicontinuous with respect to weak convergence, we infer $\vr(x) \leq \liminf \vr(x_i).$
\end{proof}

\begin{theorem}[Lower semicontinuity vs. completeness]\label{theorem:lower semicontinuity versus completeness}
 Assume that $\vr$ is symmetric and reflexive and that $(V,\mathfrak{T})$ is complete. Then the following assertions are equivalent. 
 \begin{enumerate}[(i)]
  \item $\vr$ is lower semicontinuous. 
  \item The locally convex space $(M(\vr),\mathfrak{T}_\vr)$ is complete and $\vr$ is left-continuous. 
 \end{enumerate}
 If $\mathfrak{T}_\vr$ is metrizable, then these are equivalent to:
 \begin{enumerate}[(i)]\setcounter{enumi}{2}
  \item $\vr$ is sequentially lower semicontinuous. 
 \end{enumerate}

\end{theorem}
\begin{proof}
 (i) $\Rightarrow$ (ii): As discussed above, left-continuity follows from convexity and lower semicontinuity.  Let $(x_i)$ be $\mathfrak{T}_\vr$-Cauchy. Then $(x_i)$ is also $\mathfrak{T}$-Cauchy and by the completeness of $(V,\mathfrak{T})$ there exists $x \in V$ with $\lim_i x_i = x$ with respect to $\mathfrak{T}$.  By the previous lemma $\av{\cdot}_L$ is lower semicontinuous on $(V,\mathfrak{T})$. We infer 
 $$\av{x - x_i}_L \leq \liminf_{j} \av{x_j - x_i}_L.$$
 Since $(x_i)$ is also $\av{\cdot}_L$-Cauchy, we infer $x \in M(\vr)$ and $x_i \to x$ with respect to $\av{\cdot}_L$. This implies $x_i \to  x$ with respect to $\mathfrak{T}_\vr$. 
 
 (ii) $\Rightarrow$ (i): According to the previous lemma it suffices to show the lower semicontinuity of $\av{\cdot}_L$ (considered as a functional $V \to [0,\infty]$).  Let $(x_i)$ be a net in $V$ with $x_i \to x$ with respect to $\mathfrak{T}$. We can assume $\liminf_i \av{x_i}_L < \infty $ (else there is nothing to show) and hence also $\sup_i \av{x_i}_L < \infty$ (else pass to a suitable subnet). With this at hand Lemma~\ref{lemma:weak convergence} implies $x_i \to x$ weakly in $(M(\vr),\av{\cdot}_L)$. But each seminorm is lower semicontinuous with respect to weak convergence and hence we obtain $\av{x}_L \leq \liminf \av{x_i}_L$.

 If $\mathfrak{T}_\vr$ is metrizable, then completeness of $(M(\vr),\mathfrak{T}_\vr)$ is equivalent to sequential completeness. In this case, the equivalence of (ii) and (iii) can be proven exactly as the equivalence of (i) and (ii) with nets replaced by sequences.  
 \end{proof}

 \begin{remark}
  The implication (i) $\Rightarrow$ (ii) holds without assuming the reflexivity of $\vr$. That the converse need not hold without the reflexivity assumption will be discussed below in Example~\ref{example:failure of equivalence}. 
 \end{remark}

%

\begin{theorem}\label{theorem:two imply the third}
Assume that $\vr$ is symmetric and reflexive and that $(V,\mathfrak{T})$ is complete. If $(M(\vr),\mathfrak{T}_\vr)$ is metrizable, then each two of the following imply the third. 
\begin{enumerate}[(i)]
 \item $\vr$ is (sequentially) lower semicontinuous.
 \item $\ker \av{\cdot}_L$ is $\mathfrak{T}$-closed and 
 $$(M(\vr)/\ker \av{\cdot}_L, \av{\cdot}_L) \to (V/\ker \av{\cdot}_L, \mathfrak{T}/\ker \av{\cdot}_L), \quad [x] \mapsto [x]$$
 is continuous. Here,  $\mathfrak{T}/\ker \av{\cdot}_L$ denotes the quotient topology and $[x] = x + \ker \av{\cdot}_L$
\item $(M(\vr)/\ker \av{\cdot}_L, \av{\cdot}_L)$ is a Banach space and $\vr$ is left-continuous. 
\end{enumerate}
\end{theorem}

Before proving this theorem, we recall one elementary fact about quotient topologies in locally convex (or more general topological) vector spaces. If $F \subset V$ is a subspace, then the quotient topology on $V/F$ is given by 
$$\mathfrak{T} /F  = \{F + U \mid U \in \mathfrak{T}\}.$$
For symmetric $\vr$ this implies that the quotient topology $\mathfrak{T}_\vr / \ker \av{\cdot}_L$ on $M(\vr)/\ker \av{\cdot}_L$ is the topology generated by the norm $\av{\cdot}_L$ on $M(\vr)/\ker \av{\cdot}_L$  and the restriction of the quotient topology $\mathfrak{T}/\ker \av{\cdot}_L$ to the subspace  $M(\vr)/\ker \av{\cdot}_L$. We refer to \cite[Lemma 1.39]{Schmi3}  for more details. This reference only discusses the case when $\av{\cdot}_L$ is the square root of a quadratic form, but directly extends to our situation.

\begin{proof}

 (i) \& (ii) $\Rightarrow$ (iii): Let $([x_n])$ be a Cauchy sequence in $M(\vr)/\ker \av{\cdot}_L$ with respect to $\av{\cdot}_L$. The continuity of the embedding in (ii) implies that $([x_n])$ is Cauchy in $(V/\ker \av{\cdot}_L,\mathfrak{T}/\ker \av{\cdot}_L)$. Our description of $\mathfrak{T}_\vr / \ker \av{\cdot}_L$ prior to this proof yields that $([x_n])$ is Cauchy with respect to $\mathfrak{T}_\vr / \ker \av{\cdot}_L$. According to Theorem~\ref{theorem:lower semicontinuity versus completeness}, the lower semicontinuity of $\vr$ in (i) implies that the space $(M(\vr),\mathfrak{T}_\vr)$ is complete. Since the quotient of a complete metrizable topological vector space by a closed subspace is again complete, see \cite[Theorem~6.3]{Schae71}, we infer that $M(\vr)/\ker \av{\cdot}_L$ equipped with $\mathfrak{T}_\vr / \ker \av{\cdot}_L$ is complete. Hence, there exists $x \in M(\vr)$ such that $[x_n] \to [x]$ with respect to $\mathfrak{T}_\vr / \ker \av{\cdot}_L$. Since the latter topology is finer than the $\av{\cdot}_L$-norm topology on $M(\vr) / \ker \av{\cdot}_L$, we infer $[x_n] \to [x]$ with respect to $\av{\cdot}_L$.  
 
 The left continuity of $\vr$ follows from its sequential lower semicontinuity.

 (i) \& (iii) $\Rightarrow$ (ii): According to Lemma~\ref{lemma:equivalence of lower semicontinuity}, the lower semicontinuity of $\vr$ implies the lower semicontinuity of $\av{\cdot}_L$. Hence, $\ker \av{\cdot}_L$ is $\mathfrak{T}$-closed. Now consider the bijective map 
 $$\Phi \colon (M(\vr)/\ker \av{\cdot}_L, \mathfrak{T}_\vr / \ker \av{\cdot}_L) \to (M(\vr)/\ker \av{\cdot}_L, \av{\cdot}_L), \quad  [x] \mapsto [x].$$
 By our description of  $\mathfrak{T}_\vr / \ker \av{\cdot}_L$ prior to this proof, it is continuous and it suffices to show the continuity of $\Phi^{-1}$. 
 
 Since $(M(\vr),\mathfrak{T}_\vr)$ metrizable and also complete (use Theorem~\ref{theorem:lower semicontinuity versus completeness}), the quotient space $(M(\vr)/\ker \av{\cdot}_L, \mathfrak{T}_\vr / \ker \av{\cdot}_L)$ is  metrizable and  complete, see \cite[Theorem~6.3]{Schae71}. Using also the completeness of $(M(\vr)/\ker \av{\cdot}_L, \av{\cdot}_L)$, the open mapping theorem (which holds for mappings between complete metrizable topological vector spaces, see e.g. \cite[Theorem~3.8]{Hus65}) yields the desired continuity of $\Phi^{-1}$.
 
 (ii) \& (iii) $\Rightarrow$ (i): Using Theorem~\ref{theorem:lower semicontinuity versus completeness}, it suffices to show the completeness of $(M(\vr),\mathfrak{T}_\vr)$. To this end, let $(x_n)$ be Cauchy with respect to $\mathfrak{T}_\vr$.  The completeness of $(V,\mathfrak{T})$ implies that $(x_n)$ has a $\mathfrak{T}$-limit $x$ and the completeness of $(M(\vr)/\ker \av{\cdot}_L ,\av{\cdot}_L)$ implies that it has a $\av{\cdot}_L$-limit $y$. It suffices to show $ \av{x - y}_L = 0$, as this would imply $x_n \to x$ with respect to $\av{\cdot}_L$ and hence $x_n \to x$ with respect to $\mathfrak{T}_\vr$.
 
  The continuity of the embedding in (ii) implies $[x_n]  \to [y]$ with respect to $\mathfrak{T} / \ker\av{\cdot}_L$. Moreover, the continuity of the canonical projection $(V,\mathfrak{T}) \to (V /\ker \av{\cdot}_L, \mathfrak{T}/ \av{\cdot}_L), \, z \mapsto [z],$ yields $[x_n] \to [x]$  with respect to $\mathfrak{T} / \ker\av{\cdot}_L$. Since $\ker\av{\cdot}_L$ is closed, the topology $\mathfrak{T} / \ker\av{\cdot}_L$ is Hausdorff and we infer $[x] = [y]$, i.e., $ \av{x - y}_L = 0$. 
\end{proof}

\begin{remark}
 The implications (i) \& (ii) $\Rightarrow$ (iii)  and  (i) \& (iii) $\Rightarrow$ (ii) hold without assuming  reflexivity of $\vr$. The implication (ii) \& (iii) $\Rightarrow$ (i) holds without assuming metrizability of $\mathfrak{T}_\vr$.  
\end{remark}

%
%
%
%
%

%
%
%
%
%
%
%
%
%
%
%
%
%
%
%
%
%

\section{Convex functionals on $\cF(X)$ and the elementary resistance} \label{section:elementary resistance}

From this section onwards we study convex functionals defined on all real-valued functions on a set. 

Let $X \neq \emptyset$. We write $\cF(X) = \{f \colon X \to \R\}$ for the vector space of all real-valued functions on $X$. For $x \in X$ we let $\delta_x \colon \cF(X) \to \R$, $\delta_x(f) = f(x)$. We equip $\cF(X)$ with the locally convex topology  of pointwise convergence $\mathfrak{P}$, which is induced by the family of seminorms $|\delta_x|$, $x \in X$. With this topology (which we often suppress in notation), the space $\cF(X)$ is Hausdorff and complete and so  the theory outlined in the previous section can be applied to convex functionals on $\cF(X)$.  
 
A linear functional  $\varphi \colon \cF(X) \to \R$ is continuous if and only if there exists a finite set $K \subset X$ and $C \geq 0$ such that 
$$|\varphi(f)| \leq C \sum_{x \in K} |f(x)|, \quad f \in \cF(X).$$
This allows us to identify the dual space of $\cF(X)$ with the finitely supported functions $\cF_c(X) = \{g \in \cF(X) \mid \{g \neq 0\} \text{ finite}\}$ via the dual pairing 
$$\cF_c(X) \times \cF(X) \to \R,\quad (g,f) = \sum_{x \in X} g(x) f(x).$$
As a above, we identify $(\cF_c(X),\sigma(\cF_c(X),\cF(X)))'$ with $\cF(X)$ via this dual pairing. 

For the rest of this section we assume that $\cE \colon \cF(X) \to [0,\infty]$ is a convex functional with $\cE(0) = 0$.

\subsection{The elementary resistance}

In this subsection we introduce the elementary resistance for $\cE$, which will be refined later. In particular, we discuss its relation to the Orlicz functional and, in the symmetric case, the dual norm of the Luxemburg seminorm.

\begin{definition}[Elementary resistance] \label{definition:elementary resistance}
The {\em elementary resistance} of $\cE$ is defined by
$$ R = R_\cE \colon X \times X \to [0,\infty], \quad R(x,y) = \sup \{f(x) - f(y) \mid  \cE(f) \leq 1\}.$$
Moreover, for $x \in X$ we define  {\em elementary resistance} between $x$ and $\infty$ by  
$$R_\infty(x) = R_{\infty,\cE}(x)  =  \sup \{f(x) \mid  \cE(f) \leq 1\}.$$
\end{definition}


\begin{remark}[Resistance to a boundary point]\label{remark:resistance to boundary point}
 Assume that $\Delta$ is a point  not contained in $X$. We let $\hat X = X \cup \{\Delta\}$ and consider the functional 
$$\cE_\Delta \colon \cF(\hat X) \to [0,\infty],\quad \cE_\Delta(f) = \cE(f|_X - f(\Delta)).$$
For $x \in X$ we obtain 
\begin{align*}
R_{\cE_\Delta}(x,\Delta) &= \sup \{f(x) - f(\Delta) \mid \cE_\Delta(f) \leq 1\} \\
&= \sup \{f(x) \mid \cE(f|_X) \leq 1\}\\
&= R_{\infty,\cE}(x).
\end{align*}
In this sense, $R_\infty(x)$ can be interpreted as elementary resistance to a boundary point at infinity. A similar computation shows $R_{\cE_\Delta}(x,y) = R_\cE(x,y)$ for $x,y \in X$. Moreover, the identity
$$D(\cE) = \{f|_X \mid f \in D(\cE_\Delta) \text{ with }f(\Delta) =  0\}$$
can be interpreted as $\cE$ arising from $\cE_\Delta$ by putting Dirichlet boundary conditions at $\Delta$. 
\end{remark}

\begin{proposition}[Basic properties of the elementary resistance] \label{proposition: some basic inequalities} For all $x,y,z \in X$ the following holds: 
 \begin{enumerate}[(a)]
 \item  $R(x,z) \leq R(x,y) + R(y,z)$ and $R(x,z) \leq R_\infty(x) + R_\infty(z)$. Moreover, if $R(x,y) < \infty$, then
 $$f(x) - f(y) \leq R(x,y) \av{f}_L, \quad f \in M(\cE)$$
 and if $R_\infty(x) < \infty$, then 
  $$f(x)  \leq R_\infty(x) \av{f}_L, \quad f \in M(\cE).$$

  \item If $\cE$ is symmetric, then 
  $$R(x,y) = R(y,x) = \sup \{|f(x) - f(y)| \mid \cE(f) \leq 1\}$$
  and 
  $$R_\infty(x) = \sup \{|f(x)| \mid \cE(f) \leq 1\}.$$

  \item If $\cE$ is lower semicontinuous, then 
  $$R(x,y) = \av{\delta_x - \delta_y}_{O,\cE^*} \text{ and } R_\infty(x) = \av{\delta_x}_{O,\cE^*}.$$
  
 \end{enumerate}
\end{proposition}
\begin{proof}
 (a): The first two inequalities are trivial. For the third inequality assume that $f \in M(\cE)$ and let $\lambda > \av{f}_L$. Then $\cE(f/\lambda) \leq 1$ and $|f(x)/\lambda - f(y)/\lambda| \leq R(x,y)$. Letting $\lambda \searrow \av{f}_L$ yields the claim. The fourth inequality can be proven similarly. 
 
 (b): This is trivial. 
 
 (c): Since $f(x) - f(y) = (\delta_x - \delta_y,f)$ and $f(x) = (\delta_x,f)$, both identities  follow directly from $\cE^{**} = \cE$ and the definition of the Orlicz functional.
\end{proof}

For the following corollary recall that $\av{\cdot}_{L'}$ denotes the operator norm on $(M(\cE),\av{\cdot}_L)'$.

\begin{corollary}\label{coro:elementary resistance as dual norm}
Assume that $\cE$ is symmetric and left-continuous and let $x,y \in X$. 
\begin{enumerate}[(a)]
 \item $R(x,y) < \infty$ if and only if $\delta_x - \delta_y \in M(\cE)'$ and 
 $$R(x,y) = \av{\delta_x - \delta_y}_{L'}.$$
 \item $R_\infty(x)  < \infty$ if and only if $\delta_x \in M(\cE)'$ and
 $$R_\infty(x) = \av{\delta_x}_{L'}$$
\end{enumerate}
\end{corollary}
\begin{proof}
 This follows from Proposition~\ref{proposition: some basic inequalities}~(b) and that $\cE(f) \leq 1$ if and only if $\av{f}_L \leq 1$, which follows from the left-continuity of $\cE$. 
\end{proof}

\begin{remark}
 For symmetric lower semicontinuous $\cE$ the previous corollary is just a special case of one of the identities in Corollary~\ref{coro:dual spaces}.
\end{remark}

%
%


\begin{corollary}\label{coro:kernel of e for finite elementary resistance}
 If $R(x,y) < \infty$ for all $x,y \in X$, then $\ker \av{\cdot}_L \subset \R  \cdot 1$.  
\end{corollary}
\begin{proof}
 This is a direct consequence of the inequality 
 \begin{align*}
 |f(x) - f(y)| &\leq (R(x,y) \vee R(y,x)) \av{f}_L, \quad f \in M(\cE). \hfill \qedhere
 \end{align*}
\end{proof}

 \begin{remark}[Kernel for functionals with finite elementary resistance]
Due to the positive $1$-homogeneity of $\av{\cdot}_L$, the kernel of $\av{\cdot}_L$ is a positive cone. It is even a vector space if $\cE$ is symmetric. Hence, the previous corollary shows that finiteness of the elementary resistance leaves precisely four options for $\ker \av{\cdot}_L$:  It can only be equal to $\{0\}$, $\pm [0,\infty) \cdot 1$ or $\R \cdot 1$. In the symmetric case, the only possibilities are $\{0\}$ and  $\R \cdot 1$ and at some points below we will have to distinguish between them. 
\end{remark}

Recall that if $\cE$ is symmetric, then $\mathfrak{P}_\cE$ denotes the locally convex topology generated by the seminorms $|\delta_x|$, $x \in X$, (which generate $\mathfrak P$) and $\av{\cdot}_L$.

\begin{corollary}\label{coro:metrizability form topology}
 Let $o \in X$. If $\cE$ is symmetric and $R(x,y) < \infty$ for all $x,y \in X$, then the topology $\mathfrak{P}_\cE$ on $M(\cE)$ is generated by the norm $\av{\cdot}_L  + |\delta_o|$.  In particular, the for $o' \in X$ the norms $\av{\cdot}_L  + |\delta_o|$ and $\av{\cdot}_L  + |\delta_{o'}|$ are equivalent. 
\end{corollary}
\begin{proof}
 By definition the topology $\mathfrak{P}_\cE$ is generated by the family of seminorms $|\delta_x|$, $x \in X$, and $\av{\cdot}_L$. For fixed $o \in X$ the finiteness of $R$ implies 
 $$|\delta_x(f)| = |f(x)| \leq |f(x) - f(o)| + |f(o)| \leq R(x,o) \av{f}_L + |f(o)|$$
 for all $f \in M(\cE)$. This shows that the topology $\mathfrak{P}_\cE$ is generated by the one seminorm $\av{\cdot}_L + |\delta_o|$. That it is indeed a norm follows from the previous corollary. 
 
 The 'in particular' statement follows from the fact that norms generating the same topology are equivalent.  
\end{proof}

 \begin{remark}
  This corollary can (and will) be used to produce examples where the topology $\mathfrak{P}_\cE$ is metrizable even though $\mathfrak{P}$ is not metrizable because $X$ is uncountable, see e.g. Example~\ref{example:failure of equivalence}.
 \end{remark}

\subsection{Lower semicontinuity}\label{subsection:lower semicontinuity}

In the next two theorems we characterize lower semicontinuity of symmetric and reflexive $\cE$ under the finiteness assumption on the elementary resistance. We have to treat the two possible cases $\ker \av{\cdot}_L = \R \cdot 1$ and $\ker \av{\cdot}_L = \{0\}$ separately.

\begin{theorem}\label{theorem:charcterization lower semicontinuity constant kernel}
Assume that $\cE$ is  symmetric and reflexive. Assume further $\ker \av{\cdot}_L = \R \cdot 1$ and $R(x,y) < \infty$ for all $x,y \in X$. Then the following assertions are equivalent:  
 \begin{enumerate}[(i)]
  \item $\cE$ is (sequentially) lower semicontinuous.
  \item $(M(\cE)/\R \cdot 1,\av{\cdot}_L)$ is a Banach space and $\cE$ is left-continuous.
 \end{enumerate}

\end{theorem}
\begin{proof}
We use Theorem~\ref{theorem:two imply the third}. Since by Corollary~\ref{coro:metrizability form topology} the topology $\mathfrak{P}_\cE$ is metrizable, it suffices to show the continuity of the embedding 
$$(M(\cE)/ \R \cdot 1, \av{\cdot}_L) \to (\cF(X) / \R \cdot 1, \mathfrak{P}/\R \cdot 1), \quad f + \R \cdot 1 \mapsto  f + \R \cdot 1.$$

A basis of zero neighborhoods for the quotient topology $\mathfrak{P} / \R \cdot 1$ is given by 
$$U_{K,\varepsilon} = \{f + \R \cdot 1 \mid f \in \cF(X) \text{ with } |f(x)| < \varepsilon \text{ for all } x\in K\},$$
$\varepsilon > 0$ and $K \subset X$ finite (this is the image of a basis of zero neighborhoods in $\mathfrak{P}$ under the quotient map). For given $K \subset X$ finite and fixed $o \in K$, we let $C = \max \{R(x,o) \mid x \in K\}$. If $f + \R \cdot 1 \in M(\cE)/\R \cdot 1$ with $\av{f}_L < \varepsilon/(1 + C)$, then 
$$|f(x) - f(o)| \leq R(x,o) \av{f}_L < \varepsilon$$
for all $x \in K$. Since $f + \R \cdot 1 = (f - f(o)) + \R \cdot 1$, this implies $f + \R \cdot 1 \in U_{K,\varepsilon}$ and we obtain the continuity of the embedding.  
\end{proof}

\begin{theorem}\label{theorem:charcterization lower semicontinuity trivial kernel}
Assume that $\cE$ is symmetric and reflexive. The following assertions are equivalent: 
 \begin{enumerate}[(i)]
  \item $\cE$ is lower semicontinuous with $\ker \av{\cdot}_L = \{0\}$ and $R(x,y) < \infty$ for all $x,y \in X$. 
  \item $\cE$ is left-continuous, $(M(\cE),\av{\cdot}_L)$ is a Banach space and $R_\infty(x) < \infty$ for all $x \in X$. 
 \end{enumerate}
\end{theorem}
\begin{proof}
Again we use Theorem~\ref{theorem:two imply the third}. Since $R(x,y) \leq R_\infty(x) + R_\infty(y)$, either of the Assertions (i) and (ii) imply $R(x,y) < \infty$ for all $x,y \in X$. Hence, by Corollary~\ref{coro:metrizability form topology} the topology $\mathfrak{P}_\cE$ is metrizable. According to Theorem~\ref{theorem:two imply the third}, it suffices to show that either of the assertions imply the continuity of the embedding 
$$(M(\cE),\av{\cdot}_L) \to (\cF(X),\mathfrak{P}), \quad f \mapsto f.$$
If  $R_\infty(x) < \infty$ for all $x \in X$, then  this continuity follows from the inequality $|f(x)| \leq R_\infty(x) \av{f}_L$, $f \in M(\cE)$ and $x \in X$. Since this is already an assumption in (ii),  it remains to prove the finiteness of $R_\infty$ assuming (i). 

Assume (i) and further assume there exists $x \in X$ with $R_\infty(x) = \infty$. Using $\cE(f) \leq 1$ if and only if $\av{f}_L \leq 1$, we find a sequence $(f_n)$ in $M(\cE)$ with 
$$1 = |f_n(x)| \geq n \av{f_n}_L, \, n \in \N.$$
This implies $\av{f_n}_L \to 0$ and hence 
$$|f_n(x) - f_n(y)| \leq R(x,y) \av{f_n}_L \to 0,$$
i.e., $f_n \to 1$ pointwise. Using the lower semicontinuity of $\av{\cdot}_L$ with respect to pointwise convergence (see Lemma~\ref{lemma:equivalence of lower semicontinuity}), we infer 
$$\av{1}_L \leq \liminf_{n \to \infty} \av{f_n}_L = 0,$$
a contradiction to  $\ker \av{\cdot}_L = \{0\}$.  
\end{proof}

\begin{remark}\label{remark:lower semicontinuity without reflexivitiy}
 In both theorems on lower semicontinuity the implication (i)  $\Rightarrow$ (ii) holds without reflexivity of $\cE$. Reflexivity is needed for the converse implication, as the next  example shows. Its lack of lower semicontinuity is well-known, we  provide details to show that it fits into our framework.
\end{remark}

\begin{example}[Failure of Theorem~\ref{theorem:charcterization lower semicontinuity constant kernel} without reflexivity] \label{example:failure of equivalence}
 Consider the functional 
 $$\cE \colon \cF([-1,1]) \to [0,\infty], \quad \cE(f) = \begin{cases}
                                                          \int_{-1}^1 |f'| dx &\text{if } f  \in AC([-1,1])\\
                                                          \infty &\text{else}
                                                         \end{cases}.
$$
Here, $AC([-1,1])$ denotes the space of absolutely continuous functions on the interval $[-1,1]$  and $f'$ denotes their a.s. existing derivative, which belongs to $L^1$. Then    $R(t,s) < \infty$ for all $t,s \in [-1,1]$ and $\ker \av{\cdot}_L = \R \cdot 1$. Moreover, $\av{\cdot}_L = \cE$ and  the normed space $(M(\cE),\mathfrak{P}_\cE) = (AC([-1,1]),\cE + |\delta_0|)$ is complete but $\cE$ is not lower semicontinuous.
\end{example}

\begin{proof} Due to its $1$-homogeneity, $\cE$ equals its Luxemburg seminorm and $M(\cE) = D(\cE) = AC([-1,1])$. The statement on the kernel is immediate once we prove finiteness of the elementary resistance.  For   $f \in AC([-1,1])$ the fundamental theorem of calculus for absolutely continuous functions implies 
$$|f(t) - f(s)| = \left|\int_s^t f' dx  \right|\leq \int_{-1}^1 |f'| dx, $$
and we obtain $R(t,s) \leq 1$ for all $t,s \in [-1,1]$.   This yields 
$$\av{f}_\infty \leq \cE(f) + |f(0)|$$
and that the topology $\mathfrak{P}_\cE$ on $AC([-1,1])$ is induced by the norm $\cE + |\delta_0|$. Now assume that $(f_n)$ is  $\cE + |\delta_0|$ - Cauchy. Using the previous inequality and the definition of $\cE$, we find $f \in C([-1,1])$ and $g \in L^1([-1,1])$, with $f_n \to f$ uniformly and $f_n' \to g$ in $L^1$. For $t \in [-1,1]$ this implies 
$$f(t) = \lim_{n \to \infty} f_n(t) = \lim_{n \to \infty} \left(f_n(0) + \int_0^t f_n' dx \right) = f(0) + \int_0^t g dx.$$
We obtain $f \in AC([-1,1])$ and $f_n \to f$ with respect to $\cE + |\delta_0|$. 

The lack of lower semicontinuity can be seen as follows: Consider the sequence $(f_n)$ defined by $f_n(t) = 1_{[1/n,1]}(t) + nt 1_{[0,1/n)}$. Then $f_n \in AC([-1,1])$ with $f_n' = n  1_{[0,1/n)}$ and so $\cE(f_n) \leq 1$. We also have $f_n \to 1_{(0,1]}$ pointwise but $1_{(0,1]}$ is not absolutely continuous, showing that $\cE$ is not lower semicontinuous with respect to pointwise convergence.
\end{proof}

\subsection{Approximations via finite subsets} \label{subsection:finite approximation}

In this subsection we discuss how lower semicontinuous $\cE$ can be approximated with the help of functionals on finite subsets.

For a finite set $K \subset X$, $\alpha > 0$ and $1 \leq p < \infty$ we define the approximating functionals $\cE^{(\alpha,K)} =\cE^{(\alpha,K)}_p \colon \cF(X) \to [0,\infty]$ by
$$\cE^{(\alpha,K)}(f) = \inf \{\cE(g) + \alpha \sum_{x \in K} |f(x) - g(x)|^p \mid g \in \cF(X)\}. $$
Moreover, for $f \in \cF(K)$ we let $\iota_K f \colon X \to \R$ with $\iota_K f(x) = f(x)$ for $x \in K$ and $\iota f_K(x) = 0$ for $x \in X \setminus K$.

\begin{proposition}\label{proposition:basic properties approximating forms}
\begin{enumerate}[(a)]
 \item   $\cE^{(\alpha,K)}$ is finite and continuous and the value $\cE^{(\alpha,K)}(f)$ only depends on $f|_K$. 

 \item If $\cE$ is symmetric, then $\cE^{(\alpha,K)}$ is symmetric. 
 \item If $\cE$ is positively $p$-homogeneous, then $\cE^{(\alpha,K)}_p$ is positively $p$-homogeneous.
\end{enumerate}
\end{proposition}
\begin{proof}
 (a): The finiteness of $\cE^{(\alpha,K)}(f)$ follows from the definition by letting $g = 0$ and using $\cE(0) = 0$. By definition it is clear that $\cE^{(\alpha,K)}(f)$ only depends on $f|_K$. Together, these observations imply that the functional 
 $$\cE' \colon \cF(K) \to [0,\infty),\quad f \mapsto \cE^{(\alpha,K)}(\iota_K f)$$
 is a well-defined  convex functional on the finite dimensional space $\cF(K)$ with effective domain equal to $\cF(K)$. Since any convex functional defined everywhere on a finite dimensional normed space is continuous, $\cE'$ must be continuous with respect to pointwise convergence in $\cF(K)$. Hence,  $\cE^{(\alpha,K)}$ is continuous with respect to pointwise convergence in $\cF(X)$. 
 
 (b) and (c): These are straightforward. 
\end{proof}

\begin{theorem}\label{theorem:approximating forms}
$\cE$ is lower semicontinuous with respect to pointwise convergence if and only if for some/any $1 \leq p < \infty$ we have
$$\cE = \sup \{\cE_p^{(\alpha,K)} \mid K \subset X \text{ finite  and } \alpha > 0\}. $$
\end{theorem}
\begin{proof}
If $\cE = \sup \{\cE^{(\alpha,K)} \mid K \subset X \text{ finite  and } \alpha > 0\}$, then $\cE$ is the supremum of continuous functionals and hence lower semicontinuous. 

Conversely, assume that $\cE$ is lower semicontinuous. We let $I = \{(\alpha,K) \mid \alpha > 0 \text{ and } K \subset X \text{ finite}\}$ and order these pairs by $(\alpha,K) \prec (\beta,L)$ if  $\alpha \leq \beta$ and $K \subset L$. It is straightforward that $\cE^{(\alpha,K)} \leq \cE$ for all $(\alpha,K) \in I$ and we only have to prove that the supremum is at least   $\cE$. 

Case 1: $\sup_{(\alpha,K)\in I} \cE^{(\alpha,K)} (f)  = \infty$. In this case, we obtain $\cE(f) =  \infty$ and equality is proven. 

Case 2: $C := \sup_{(\alpha,K)\in I} \cE^{(\alpha,K)} (f) < \infty$. For each $(\alpha,K) \in I$ we choose $f_{(\alpha,K)} \in \cF(X)$ with 
$$\cE(f_{(\alpha,K)}) + \alpha \sum_{x \in K} |f(x) - f_{(\alpha,K)}(x)|^p \leq \cE^{(\alpha,K)}(f) + \frac{1}{\alpha}.$$
For $\alpha \geq 1$ and $K \subset X$ with $\{x\} \subset K$, this implies 
$$  |f(x) - f_{(\alpha,K)}(x)|^p \leq \frac{C + 1}{\alpha}.$$
In particular, the net $(f_{(\alpha,K)})$ converges pointwise to $f$ and the lower semicontinuity of $\cE$ implies
$$\cE(f) \leq \liminf_{(\alpha,K) \in I} \cE(f_{(\alpha,K)}) \leq  \liminf_{(\alpha,K) \in I} \left(\cE^{(\alpha,K)}(f) + \frac{1}{\alpha}\right) \leq C. $$
\end{proof}

\begin{corollary}[Approximation via finite sets] \label{corollary:approximation via finite subsets}
 $\cE$ is lower semicontinuous if and only if there exist $I \neq \emptyset$ and finite sets $K_i$, continuous convex functionals $\cE_i \colon \cF(K_i) \to [0,\infty)$ with $\cE_i(0) = 0$, and  injective $\Phi_i \colon K_i \to X$, $i \in I$, such that 
 $$\cE(f) = \sup_i \cE_i (f  \circ \Phi_i), \quad f \in \cF(X). $$
 %
 %
\end{corollary}
\begin{proof}
The 'only if'-part follows directly from the previous theorem. For the 'if'-part we note that the functional $\cE_i' \colon \cF(X) \to [0,\infty)$, $\cE_i'(f) = \cE_i(f \circ \Phi_i)$ is continuous. Hence, $\cE$ is lower semicontinuous as supremum of continuous functionals. 
\end{proof}

\section{Resistance forms and the resistance metric} \label{section:nonlinear resistance forms}

For symmetric functionals the elementary resistance is an extended pseudo metric but it does not satisfy the main feature of resistance, namely the additivity over series circuits. For this reason, we introduce another family of extended pseudometrics with the additivity feature. In order to  do so, we need to additionally assume the compatibility of $\cE$ with normal contractions,  leading to the notion of nonlinear resistance forms.  

As in the previous section, $\cE \colon \cF(X) \to [0,\infty]$ is assumed to be convex with $\cE(0) = 0$.

\subsection{The resistance} \label{subsection:resistance}

\begin{definition}[Resistance] \label{definition:resistance}
For $t > 0$ we define the {\em $t$-resistance} of $\cE$ by
$$\cR_t \colon X \times X \to [0,\infty],\quad \cR_t(x,y) = \sup \{ t(f(x) - f(y)) - \cE(f) \mid f \in \cF(X)\}.$$
Moreover, for $x \in X$ we define the $t$-resistance between $x \in X$ and $\infty$ by 
$$\cR_{t,\infty}(x) = \sup \{ t f(x) - \cE(f) \mid f \in \cF(X)\}.$$
\end{definition}

\begin{remark}
 As for the elementary resistance, the $t$-resistance between points and infinity can be interpreted as the $t$-resistance for a modified functional on an enlarged space, cf. Remark~\ref{remark:resistance to boundary point}.
\end{remark}

%
%
%
%

\begin{proposition}[Elementary properties of resistance] \label{prop:elementary properties resistance}
Let $x,y \in X$. 
\begin{enumerate}[(a)]
 \item The maps $(0,\infty) \to [0,\infty]$, $t \mapsto \cR_t(x,y)$ and $(0,\infty) \to [0,\infty]$, $t \mapsto \cR_{t,\infty}(x)$  are convex and lower semicontinuous.
 \item If $\cE$ is symmetric, then for each $t > 0$ we have 
 $$\cR_t(x,y) = \sup \{ t |f(x) - f(y)| - \cE(f) \mid f \in \cF(X)\}$$
 and 
 $$ \cR_{t,\infty}(x) = \sup \{ t |f(x)|  - \cE(f) \mid f \in \cF(X)\}.$$
 %
 
 %
 
 \item $\cR_t(x,y) = \cE^*(t (\delta_{x} - \delta_{y}))$ and $\cR_{t,\infty}(x) = \cE^*(t \delta_x)$.
 
 \item $t R(x,y) \leq 1 + \cR_t(x,y)$ and $t R_{\infty}(x) \leq 1 + \cR_{t,\infty}(x)$.
 %
\end{enumerate}
\end{proposition}
\begin{proof}
(a): The supremum of affine linear functions is convex and lower semicontinuous. 

(b): This is trivial. 

(c): This follows directly from the definition of $\cE^*$.

(d): We compute using the definitions
\begin{align*}
    \cR_t(x,y) &=    \sup \{t (f(x) - f(y)) - \cE(f) \mid f \in \cF(X)\}\\
    &\geq \sup \{t (f(x) - f(y)) - \cE(f) \mid \cE(f) \leq 1\}\\
    &\geq  t\sup \{ f(x) - f(y) \mid \cE(f) \leq 1\} - 1\\
    &= t R(x,y) - 1.
    \end{align*}
The inequality for $R_\infty$ can be inferred similarly. 
\end{proof}

\begin{corollary}[Finiteness of the resistance]\label{coro:finiteness of resistance}
 Assume that $\cE$ is lower semicontinuous and let  $x,y \in X$. The following assertions are equivalent: 
 \begin{enumerate}[(i)]
  \item $\delta_{x} - \delta_y \in M(\cE^*)$. 
  \item $R(x,y)  < \infty$. 
  \item   $\cR_t(x,y) < \infty$ for some $t > 0$. 
 \end{enumerate}
%
 \end{corollary}
\begin{proof}
 (i) $\Leftrightarrow$ (ii): Proposition~\ref{proposition: some basic inequalities} shows 
 $$R(x,y) = \av{\delta_{x} - \delta_y}_{O,\cE^*}.$$
 Moreover, by Proposition~\ref{proposition:fundamental inequalities} we have $\av{\delta_{x} - \delta_y}_{O,\cE^*} < \infty$ if and only if $\delta_{x} - \delta_y \in M(\cE^*)$. 
 
 (i) $\Leftrightarrow$ (iii): By definition we have $\delta_{x} - \delta_y \in M(\cE^*)$ if and only if $t (\delta_{x} - \delta_y) \in D(\cE^*)$ for some $t > 0$. Moreover, the previous proposition shows $\cE^*(t (\delta_{x} - \delta_y)) = \cR_t(x,y)$ and so the claim follows.
%
%
%
%
 %
\end{proof}

With exactly the same arguments as for the previous corollary, we obtain the following. 

\begin{corollary}[Finiteness of the resistance to infinity]\label{coro:finiteness of resistanceII}
 Assume that $\cE$ is lower semicontinuous and let  $x \in X$. The following assertions are equivalent: 
 \begin{enumerate}[(i)]
  \item $\delta_{x} \in M(\cE^*)$. 
  \item $R_\infty(x)  < \infty$. 
  \item  $\cR_{t,\infty}(x) < \infty$ for some $t > 0$. 
 \end{enumerate}
%
 \end{corollary}

  The elementary resistance is the Orlicz functional of $\cE^*$ applied to $\delta_x - \delta_y$, see Proposition~\ref{prop:elementary properties resistance}, whereas the $t$-resistance is the convex conjugate  $\cE^*$ applied to $t(\delta_x - \delta_y)$. In general it is hard to compute either of them directly.  For homogeneous functionals however, we know that both are related with the Luxemburg seminorm of $\cE^*$ through a direct formula. This leads to the following result.
 
%

\begin{theorem}[Resistance for homogeneous functionals] \label{theorem:homogeneous resistance}
 Assume that $\cE$ is lower semicontinuous and positively $p$-homogeneous for some $1 \leq p < \infty$. Moreover, let $p^{-1} + q^{-1}  = 1$. If $p > 1$, then 
 $$\cR_t(x,y) = (p-1) \left(\frac t p\right)^q R(x,y)^q  $$
 and 
 $$\cR_{t,\infty}(x) = (p-1) \left(\frac t p\right)^q R_\infty(x)^q.$$
 If $p = 1$, then 
 $$\cR_t(x,y) = \begin{cases}
                 0 &\text{if } R(x,y) \leq \frac{1}{t}\\
                 \infty &\text{if } R(x,y) > \frac{1}{t}
                \end{cases} \text{ and } \cR_{t,\infty}(x) = \begin{cases}
                 0 &\text{if } R_\infty(x) \leq \frac{1}{t}\\
                 \infty &\text{if } R_\infty(x) > \frac{1}{t}
                \end{cases}.$$

\end{theorem}
\begin{proof}
It suffices to treat $\cR_t$, the statements on $\cR_{t,\infty}$ follow with similar arguments. First assume that $p > 1$. As seen in Proposition~\ref{prop:seminorms for homogenous functionals}, the convex conjugate $\cE^*$ is positively $q$-homogeneous. Hence, $M(\cE^*) = D(\cE^*)$ and Proposition~\ref{prop:elementary properties resistance}~(c) together with Corollary~\ref{coro:finiteness of resistance} show  $\cR_t(x,y) < \infty$ if and only if $R(x,y) < \infty$. Therefore, we can assume $\delta_{x} - \delta_y \in M(\cE^*)$. With the help of   Proposition~\ref{prop:elementary properties resistance}  and  Proposition~\ref{prop:seminorms for homogenous functionals} applied to $\cE^*$ we obtain 
\begin{align*}
 \cR_t(x,y) &= \cE^*(t(\delta_x - \delta_y)) = t^q \cE^*(\delta_x - \delta_y) = t^q \av{\delta_x - \delta_y}^q_{L,\cE^*} \\
 &= t^q \frac{p-1}{p^q} \av{\delta_x - \delta_y}^q_{O,\cE^*} = (p-1) \left(\frac t p\right)^q R(x,y)^q.
\end{align*}

Now let $p = 1$. Proposition~\ref{prop:seminorms for homogenous functionals} shows $\cE^*(t(\delta_x - \delta_y)) = 0$ if $t(\delta_x - \delta_y) \leq \av{\cdot}_L$ (as functionals on $\cF(X)$) and $\cE^*(t(\delta_x - \delta_y)) = \infty$ else. But by the $1$-homogeneity the inequality $t(\delta_x - \delta_y) \leq \av{\cdot}_L$  is equivalent to 
$$t (f(x) - f(y)) \leq \av{f}_L = \cE(f)$$
for all $f \in D(\cE)$, which in turn is equivalent to $R(x,y) \leq t^{-1}$.
\end{proof}

 \begin{remark}\label{remark:comparing resistances}
  The resistance forms of Kigami \cite{Kig03} are $2$-homogenous. The resistance he considers equals $R^2$ in our notation. Hence, our theorem  applied to $p = 2$ shows that $\mathcal R_2 = R^2$ equals Kigami's resistance. For $1 <  p < \infty$, the $p$-resistance forms of Kajino and Shimizu in \cite{KS25} are $p$-homogeneous functionals with finite elementary resistance (see Section~\ref{subsection:p-resistance forms} for a precise definition). The resistance they consider equals $R^p$ in our notation. Our theorem relates these quantities through the identity
  $$\mathcal R_p = (p-1) R^q = (p-1) (R^p)^\frac{1}{p-1}.$$
 \end{remark}

 \subsection{Nonlinear resistance forms and compatibility with normal contractions}

 While the elementary resistance trivially satisfies the triangle inequality,   the $t$-resistance only satisfies it if $\cE$ is compatible with certain normal contractions. This is discussed next. 
 
 Recall that a {\em normal contraction}  is a $1$-Lipschitz function $C \colon \R \to \R$ with $C(0) = 0$. We say that the normal contraction $C$ {\em operates on $\cE$} or that {\em $\cE$ is compatible with $C$} if 
 $$\cE(f + C g) + \cE(f - Cg) \leq \cE(f + g) + \cE(f - g)$$
 for all $f,g \in  \cF(X)$. 
 It is straightforward that $C$ operates on $\cE$ if and only if $-C$ operates on $\cE$. In particular, ${\rm id}$ and $-{\rm id}$ operate on $\cE$.

 In the following lemma we use the approximating functionals discussed in Subsection~\ref{subsection:finite approximation} with some fixed $p \geq 1$.

\begin{lemma}\label{lemma:reduction to continuous case}
 Assume that $C$ is lower semicontinuous and let $C \colon \R \to \R$ be a normal contraction. The following assertions are equivalent. 
 \begin{enumerate}[(i)]
  \item $C$ operates on $\cE$. 
  \item For all $\alpha > 0$ and finite $K \subset X$ the normal contraction $C$ operates on $\cE^{(\alpha,K)}$. 
 \end{enumerate}
\end{lemma}
\begin{proof}
(ii) $\Rightarrow$ (i): This follows directly form $\cE = \sup_{(\alpha,K)}  \cE^{(\alpha,K)}$ and the fact that $(\alpha,K) \to \cE^{(\alpha,K)}$ is increasing (with respect to the obvious order, which was also discussed in the proof of Theorem~\ref{theorem:approximating forms}).

 (i) $\Rightarrow$ (ii): If $C \colon \R \to \R$ is a normal contraction, then for $g,h' \in \cF(X)$ and $x \in X$ we have $Cg(x) - Ch'(x) = \lambda_x (g(x) - h'(x))$ for some $|\lambda_x| \leq 1$.   Using this identity, Lemma~\ref{lemma:convexity estimate} applied to $t \mapsto |t|^p$ and that $\cE$ is compatible with $C$, for $f,h \in \cF(X)$ we obtain 
 \begin{align*}
 &\cE(h + h') + \alpha \sum_{x \in K} |f(x) + g(x)  - h(x) - h'(x)|^p\\
 &\quad + \cE(h - h') + \alpha \sum_{x \in K} |f(x) - g(x)  - h(x) + h'(x)|^p\\
  &\geq \cE(h + Ch') + \alpha \sum_{x \in K} |f(x) + Cg(x)  - h(x) - Ch'(x)|^p\\
  &\quad + \cE(h - Ch') + \alpha \sum_{x \in K} |f(x) - Cg(x)  - h(x) + Ch'(x)|^p\\
  &\geq \cE^{(\alpha,K)}(f + Cg) + \cE^{(\alpha,K)}(f- Cg).
 \end{align*}
 Since $h,h'$ were arbitrary, taking the infimum over them yields
 $$\cE^{(\alpha,K)}(f + g) + \cE^{(\alpha,K)}(f- g) \geq \cE^{(\alpha,K)}(f + Cg) + \cE^{(\alpha,K)}(f- Cg),$$
 i.e., $C$ operates on $\cE^{(\alpha,K)}$.
\end{proof}

The next theorem characterizes when all normal contractions operate on $\cE$. It can be seen as an analogue to the second Beurling-Deny criterion for Dirichlet forms. 
  
\begin{theorem}\label{proposition:reduction contraction}
Assume that $\cE$ is lower semicontinuous. The following assertions are equivalent: 
\begin{enumerate}[(i)]
 \item All normal contractions operate on $\cE$. 
 \item The family of normal contractions 
 $$C_\alpha \colon \R \to \R, \quad x \mapsto x \wedge \alpha, \quad \alpha > 0,$$
 operates on $\cE$. 
 \item The family of normal contractions 
$$D_\beta \colon\R \to \R, \quad x \mapsto |x -\beta| - |\beta|, \quad \beta \in \R,$$
 operates on $\cE$. 
\end{enumerate}
If for some $1 \leq p < \infty$ the functional $\cE$ is positively $p$-homogeneous, then these are equivalent to 
\begin{enumerate}[(i)]  \setcounter{enumi}{3}
 \item $C_1 \colon \R \to \R$, $x \mapsto x \wedge 1$, operates on $\cE$. 
 \item $D_1 \colon\R \to \R$, $x \mapsto |x -1| - 1$ operates on $\cE$. 
\end{enumerate}
\end{theorem}

 \begin{remark}
  In the context of lower semicontinuous convex functionals on $L^2(X,\mu)$ the equivalence of (i) and (ii) in the previous theorem was recently obtained in \cite[Theorem 2]{Puc25}. The corresponding functionals are called nonlinear Dirichlet forms.  Indeed, we use our approximations via functionals on functions on finite sets to reduce the equivalence to this case. 
 \end{remark}

 \begin{proof}
  (i) $\Leftrightarrow$ (ii): According to Lemma~\ref{lemma:reduction to continuous case} and Proposition~\ref{proposition:basic properties approximating forms}, it suffices to consider the case where $X$ is finite and $\cE$ is continuous. If we let $\mu$ be the counting measure on all subsets of $X$, then in this case $\cF(X) = L^2(X,\mu)$ and the topologies of pointwise convergence and $L^2$-convergence agree. Hence, the equivalence (i) and (ii) follows from \cite[Theorem~2]{Puc25}, which treats (lower semi)continuous convex functionals on $L^2$-spaces. 
  
  (i) $\Rightarrow$ (iii): This is clear.
  
  (iii) $\Rightarrow$ (ii): It is readily verified that for $\alpha > 0$ we have
  $$C_\alpha = \frac{1}{2} {\rm id} + \frac{1}{2} (-D_\alpha).$$
  Since $\cE$ is convex and $-D_\alpha$ operates on $\cE$, we obtain the claim. 
  
  Now assume that $\cE$ is positively $p$-homogeneous. (i) $\Rightarrow$ (v) is obvious and (v) $\Rightarrow$ (iv) can be proven as (iii) $\Rightarrow$ (ii). 
  
  (iv) $\Rightarrow$ (ii): If we let $\tilde f  = \alpha^{-1} f$, this implication follows from 
  \begin{align*}
  \cE(f + (g \wedge \alpha)) + \cE(f - (g \wedge \alpha)) &= \alpha^p\cE(\tilde f + (g \wedge 1)) + \alpha^p \cE(\tilde f- (g \wedge 1))\\
  &\leq \alpha^p\cE(\tilde f + g ) + \alpha^p \cE(\tilde f - g)\\
  &= \cE(f + g) + \cE(f - g).   \hfill \qedhere
  \end{align*}

  \end{proof}

 With all of these preparations we can now give the main definition of this paper. 
 
\begin{definition}[Nonlinear resistance form]\label{definition:nonlinear resistance form}
 A functional $\cE \colon \cF(X) \to [0,\infty]$ is called {\em nonlinear resistance form} if the following are satisfied. 
\begin{enumerate}[(nRF1)]
 \item $\cE$ is convex with $\cE(0) = 0$. 
 \item $\cE$ is lower semicontinuous with respect to pointwise convergence.
 \item For all $x,y \in X$ we have $R(x,y) < \infty$. 
 \item All normal contractions operate on $\cE$. 
\end{enumerate}
\end{definition}

\begin{remark}
\begin{enumerate}[(a)]
 \item Assume that $\cE$ is a lower semicontinuous quadratic form, i.e., it is $2$-homogeneous and satisfies the parallelogram identity. Then the compatibility with all normal contractions is reduced to the compatibility with $x \mapsto x \wedge 1$, see Theorem~\ref{proposition:reduction contraction}. Moreover, the parallelogram identity yields  $\cE(f + g \wedge 1) + \cE(f - g \wedge 1) = 2 \cE(f) + 2\cE(g \wedge 1)$ and $\cE(f + g) + \cE(f - g) = 2\cE(f) + 2\cE(g)$. Hence, in this case, (nRF4) is equivalent to $\cE(f \wedge 1) \leq \cE(f)$ for all $f \in \cF(X)$. In particular, this shows that resistance forms in the sense of Kigami \cite{Kig03} are nonlinear resistance forms (the lower semicontinuity of resistance forms in the sense of Kigami \cite{Kig03} follows from Theorem~\ref{theorem:charcterization lower semicontinuity constant kernel}).  In Proposition~\ref{proposition:p resistance} below we shall see that their $p$-homogeneous counterparts, the $p$-resistance forms recently introduced by Kajino and Shimizu \cite{KS25}, also fall into our framework.
 
 \item Usually, when dealing with resistance forms, one makes slightly different assumptions. More precisely, one (additionally) assumes the following: 
 \begin{enumerate}[(1)]
  \item   $\ker \cE = \R \cdot 1$   (which would have to be replaced by $\ker \av{\cdot}_L = \R \cdot 1$ in the fully nonlinear situation). 
  \item  $D(\cE)$ separates the points of $X$.
  \item $D(\cE)/\R \cdot 1$ is complete with respect to an appropriate norm (which would have to be replaced by the completeness of $(M(\cE)/\R \cdot 1,\av{\cdot}_L)$ in the fully nonlinear situation). 
 \end{enumerate}

 Here, we do not assume (1) and (2), because they are not necessary for the purposes of our paper and they exclude the following natural construction: Given a nonlinear resistance form $\cE$ and $F \subset X$, we consider  
 $$\cE_F \colon \cF(X) \to [0,\infty], \quad \cE_F(f) = \begin{cases}
                                                         \cE(f) &\text{if } f =  0 \text{ on }F\\
                                                         \infty &\text{else}
                                                        \end{cases}.
$$
In our terminology $\cE_F$  is a nonlinear resistance form, whereas if $F \neq \emptyset$, it does not satisfy (1),  and if $F$ contains more than $2$ points, it does not satisfy  (2). The functional $\cE_F$ corresponds to imposing 'Dirichlet boundary conditions' at $F$. 

Regarding (3), we discussed in Subsection~\ref{subsection:lower semicontinuity} how for reflexive and symmetric functionals lower semicontinuity is related to the completeness of the modular space. In view of our use of convex analysis, we believe that lower semicontinuity is the more natural assumption in the fully nonlinear case. For resistance forms in the sense of Kigami \cite{Kig03} and $p$-resistance forms (with $1 < p < \infty$) in the sense of Kajino and Shimizu \cite{KS25}, lower semicontinuity is equivalent to completeness of the modular space, see also the discussion in Subsection~\ref{subsection:p-resistance forms}. We would like to stress however, that with our notion one could also treat the case $p=1$, which is not covered by \cite{KS25}.
\end{enumerate}
 
\end{remark}


 \subsection{The triangle inequality for the resistance and additivity over serial circuits}
 
 In this subsection we give the justification why we study the resistance $\cR_t$ instead of $R$: It is an extended pseudometric that is additive  over serial circuits.

\begin{theorem}[Triangle inequality for the resistance]\label{theorem:triangle inequality}
Let $\cE$ be a nonlinear resistance form.  
\begin{enumerate}[(a)]
 \item The triangle inequality $\cR_t(x,z) \leq \cR_t(x,y) + \cR_t(y,z)$ holds for all $x,y,z \in X$ and $t > 0$. 
 \item If $\cE$ is symmetric, then  $\cR_t(x,y) = \cR_t(y,x)$ for all $x,y \in X$ and $t > 0$.
 \item If $\cE$ satisfies the $\nabla_2$-condition, then $\cR_t(x,y) < \infty$ for all $x,y \in X$ and $t> 0$.
\end{enumerate}
In particular, if $\cE$ is symmetric and satisfies the $\nabla_2$-condition, then for each $t > 0$ the function $\cR_t$ is a pseudometric. 
\end{theorem}
\begin{proof}
 (a):  We need to show
 $$ T: = t(f(x) - f(z)) - \mathcal{E}(f) \leq \mathcal{R}_t(x,y) + \mathcal{R}_t(y,z) $$
 for all $f \in D(\cE)$. We have the trivial identity
  \begin{align*}
 t(f(x) - f(z)) - \mathcal{E}(f)  & = t(f(x) - f(y)) + t(f(y) -  f(z)) - \mathcal{E}(f).
 \end{align*}
 Hence, if $f(y) > f(x)$, then $T \leq \cR_t(y,z)$ and if $f(y) < f(z)$, then $T \leq \cR_t(x,y)$.
 
 It remains to treat the case $f(z) \leq f(y) \leq f(x)$.  Let $\alpha = f(y)$ and consider the normal contraction 
   $$ C\colon \R\to\R, \quad C(t) = |t - \tfrac{\alpha}{2}| - |\tfrac{\alpha}{2}|. $$
 Using the compatibility of $\cE$ with normal contractions, we obtain
 \begin{align*}
  \mathcal{E}(f) & = \mathcal{E}(f) + \mathcal{E}(0) \\
                 & = \mathcal{E}(\tfrac{f}{2} + \tfrac{f}{2}) + \mathcal{E}(\tfrac{f}{2} - \tfrac{f}{2}) \\
                 & \geq \mathcal{E}(\tfrac{f}{2} + C(\tfrac{f}{2})) + \mathcal{E}(\tfrac{f}{2} - C(\tfrac{f}{2})) \\
                 & = \mathcal{E}(f\vee \alpha - \alpha_+) + \mathcal{E}(f\wedge \alpha + \alpha_-).
 \end{align*}
 Hence, we can estimate 
 \begin{align*}
  T & \leq  t(f(x) - f(y)) - \mathcal{E}(f\vee \alpha - \alpha_+)   + t(f(y) - f(z)) - \mathcal{E}(f\wedge \alpha + \alpha_-) \\
    & \leq \mathcal{R}_t(x,y) + \mathcal{R}_t(y,z).
 \end{align*}
The last inequality follows from the fact that $g = f \vee \alpha - \alpha_+$ and $h = f\wedge \alpha + \alpha_-$  satisfy $f(x) - f(y) = g(x) - g(y)$ and $f(y) - f(z) = h(y) - h(z)$ because $f(z) \leq \alpha = f(y) \leq f(x)$. 

(b): This is trivial.

(c): By  Lemma~\ref{lemma:duality delta and nabla} the $\nabla_2$-condition is equivalent to the $\Delta_2$-condition for $\cE^*$. In particular, $D(\cE^*)$ is a positive cone and we obtain $t (\delta_x - \delta_y) \in D(\cE^*)$ for all $t > 0$. 
%
%
%
%
%
%
%
\end{proof}

\begin{remark}\label{remark:triangle inequality and contractions}
  Except in (c), the proof of the theorem does not use lower semicontinuity. Moreover, it does not use compatibility with all normal contractions, but only with the very specific family $D_\beta \colon \R \to \R, \, t \mapsto |t - \beta| - |\beta|$, $\beta \in \R$. However, in Theorem~\ref{proposition:reduction contraction} we saw that for lower semicontinuous functionals the compatibility with the family $D_\beta$, $\beta \in \R$, is equivalent to the compatibility with all normal contractions. In this sense, the compatibility with all normal contractions that we assume is the minimal assumption required for our proof of the triangle inequality for $\cR_t$.
  \end{remark}

\begin{corollary}
 If $1 < p < \infty$, $q^{-1} + p^{-1} = 1$ and $\cE$ is a symmetric positively $p$-homogeneous nonlinear resistance form, then $R^q$ is a pseudometric. 
\end{corollary}
\begin{proof}
Under the assumptions $\cE^*$ is $q$-homogeneous and so $\cE$ satisfies the $\nabla_2$-condition. Hence, the previous theorem shows that for all $t> 0$ the resistance $\cR_t$ is a pseudometric. Since by Theorem~\ref{theorem:homogeneous resistance} $R^q$ is a constant multiple of  $\cR_t$, we obtain that $R^q$ is a pseudometric.
\end{proof}

\begin{remark}
In \cite{KS25} $p$-resistance forms are considered with $1 < p < \infty$, which are certain $p$-homogeneous convex functionals having a strong compatibility with normal contractions (see Subsection~\ref{subsection:p-resistance forms} for a precise definition). The fact that for such forms $R^q$ is a pseudometric is contained in \cite[Corollary~6.32]{KS25} (note that the quantity $R_\cE$ in  \cite{KS25} corresponds to $R^p$ in our notation -  this explains the difference of the exponents, see also Remark~\ref{remark:comparing resistances}). 

We believe that our approach to proving the triangle inequality for $R^q$ is somewhat more transparent: The triangle inequality for $\cR_t$ has a relatively simple proof and for $p$-homogeneous functionals $\cR_t$ and $R^q$ are the same up to a constant.  
\end{remark}

Next we discuss why the resistance is additive over serial circuits. There are two possible ways of producing serial circuits. Given two resistance forms and two points in their underlying space, one can either identify these two points directly or connect them by a single resistor and then let the strength of this resistor tend to $0$.   

\begin{theorem}[Additivity over serial circuits - I] \label{theorem:additivity serial circuits}
 Let $\cE_i$, $i = 1,2$ be nonlinear resistance forms on disjoint sets $X_i$, $i = 1,2$, with $\ker \cE_i \supset \R \cdot 1$. For fixed $\xi_i \in X_i$, $i = 1,2$, we define $\cE \colon \cF(X_1 \cup X_2) \to [0,\infty]$ by
 $$\cE(f) = \begin{cases}
                                                               \cE_1(f|_{X_1}) +  \cE_2(f|_{X_2}) &\text{if } f|_{X_1}(\xi_1) =  f|_{X_2}(\xi_2)\\
                                                               \infty &\text{else}
                                                               \end{cases}.
$$
 Then $\cE$ is a nonlinear resistance form with $\ker \cE \supset \R \cdot 1$ and for $x_i \in X_i$, $i = 1,2$, and $t > 0$  we have
 $$\cR_{t,\cE}(x_1,x_2) = \cR_{t,\cE_1}(x_1,\xi_1)  + \cR_{t,\cE_2}(\xi_2,x_2).$$
\end{theorem}

\begin{remark}
 \begin{enumerate}[(a)]
  \item The additivity over serial circuits seems to be a new result for nonlinear resistance forms. In \cite{KS25} it was only shown for a very specific example \cite[Example 6.34]{KS25}.
  \item The convexity of $\cE$ implies $\ker \av{\cdot}_L \subset \ker \cE$. Hence, the previous theorem can be applied to resistance forms with $\ker \av{\cdot}_L = \R \cdot 1$. 
 \end{enumerate}
\end{remark}

Before proving the theorem we need a little lemma. 

\begin{lemma}\label{lemma:slightly different formula resistance}
 Assume that $\cE$ is lower semicontinuous and satisfies $\cE(K) = 0$ for each constant function $K \in \R$. Then $\cE(f + K) = \cE(f)$ for all $f \in \cF(X)$. In particular, for each $z  \in X$ we have
 $$\cR_t(x,y) = \sup \{t (f(x) - f(y)) - \cE(f) \mid f \in D(\cE) \text{ with } f(z) = 0\}.$$
\end{lemma}
\begin{proof}
 Using convexity and lower semicontinuity, we obtain 
 \begin{align*}
 \cE(f + K) &\leq \liminf_{\varepsilon \to 0+} \cE((1-\varepsilon) f + \varepsilon K/\varepsilon)\\
  &\leq \liminf_{\varepsilon \to 0+} \left((1-\varepsilon) \cE(f) + \varepsilon \cE(K/\varepsilon)\right)\\
  &=\cE(f).  
 \end{align*}
 The formula for the resistance then follows from $\cE(f) = \cE(f - f(z))$. 
\end{proof}

\begin{proof}[Proof of Theorem~\ref{theorem:additivity serial circuits}]
It is clear that $\cE$ is convex with $\cE(0) = 0$, lower semicontinuous and that each normal contraction operates on $\cE$. For the finiteness of the elementary resistance we let $f \in \cF(X_1 \cup X_2)$ with $\cE(f) \leq 1$. Then $f(\xi_1) = f(\xi_2)$ and for $i =1,2$ we have $\cE_i(f|_{X_i}) \leq 1$. Hence, if $x,y \in X_i$, then $f(x) - f(y) = f|_{X_i}(x) - f|_{X_i}(y) \leq R_{\cE_i}(x,y)$ and if $x \in X_1$ and $y \in X_2$, then 
$$f(x) - f(y)  = f(x) - f(\xi_1) + f(\xi_2) - f(y) \leq R_{\cE_1}(x,\xi_1) + R_{\cE_2}(\xi_2,y). $$
Taking the supremum over such $f$ yields finiteness of the elementary resistance.   
 
Additivity of the resistance: By the previous lemma, in the definition of $\cR_{t,\cE}$ we can take the supremum over functions $f \in D(\cE)$ with $f(\xi_1)  = 0$. By the definition of $\cE$, such functions also satisfy $f(\xi_2)= 0$. Since for any $f_1 \in D(\cE_1)$ and $f_2 \in D(\cE_2)$ with $f_1(\xi_1) = f_2(\xi_2) = 0$ the composite function $f(x) = f_i(x)$ if $x \in X_i$, $i = 1,2$, belongs to $D(\cE)$, this implies

 \begin{align*}
 &\cR_t(x,y) = \sup\{t(f(x) - f(y)) - \cE(f) \mid f \in D(\cE) \text{ with } f(\xi_1) =  0 \}\\
 &= \sup\{tf(x) - \cE_1(f|_{X_1}) + t(-  f(y)) - \cE_2(f|_{X_2}) \mid f \in D(\cE) \text{ with } f(\xi_i) =  0\}\\
 &= \sup\{tf_1(x) - \cE_1(f_1) \mid f_1 \in D(\cE_1) \text{ with } f_1(\xi_1) = 0\}\\
 &\quad + \sup\{t (- f_2(y)) - \cE_2(f_2) \mid f_2 \in D(\cE_2) \text{ with } f_2(\xi_2) = 0\}\\
 &= \cR_{t,\cE_1}(x_1,\xi_1)  + \cR_{t,\cE_2}(\xi_2,x_2).
 \end{align*}
For the last equality we used Lemma~\ref{lemma:slightly different formula resistance} again. 
\end{proof}

The following theorem is  an alternative way of stating the additivity property.

\begin{theorem}[Additivity over serial circuits - II] \label{theorem:additivity serial circuitsII}
 Let $\cE_i$, $i = 1,2$ be nonlinear resistance forms on disjoint sets $X_i$, $i = 1,2$, with $\ker \cE_i \supset \R \cdot 1$. For fixed $\xi_i \in X_i$, $i = 1,2$, and $\varepsilon > 0$ we define $\cE  \colon \cF(X_1 \cup X_2) \to [0,\infty]$ by
 $$\cE(f) = \cE_1(f|_{X_1}) + \cE_2(f|_{X_2}) + \frac{1}{4\varepsilon} |f(\xi_1) - f(\xi_2)|^2. 
$$
 Then $\cE$ is a nonlinear resistance form with $\ker \cE \supset \R \cdot 1$ and for $x_i \in X_i$, $i = 1,2$, and $t > 0$  we have
 $$\cR_{t,\cE}(x_1,x_2) = \cR_{t,\cE_1}(x_1,\xi_1)  + \cR_{t,\cE_2}(\xi_2,x_2) + \varepsilon t^2.$$
\end{theorem}

\begin{proof}
The fact that $\cE$ is a nonlinear resistance form can be established as in the proof of the previous theorem, we refrain from giving details. 

The definition of $\cE$ and $\ker \cE_i = \R \cdot 1$ yield   $\cR_{t,\cE}(x_1,\xi_1) = \cR_{t,\cE_1}(x_1,\xi_1)$ and $\cR_{t,\cE}(\xi_2,x_2) = \cR_{t,\cE_2}(\xi_2,x_2)$. With this at hand, the triangle  inequality for $\cR_{t,\cE}$ implies  
$$\cR_{t,\cE}(x_1,x_2) \leq \cR_{t,\cE_1}(x_1,\xi_1) + \cR_{t,\cE_2}(\xi_2,x_2) + \cR_{t,\cE}(\xi_1,\xi_2).$$
Moreover, using Example~\ref{example:basic duality}, we have 
\begin{align*}
\cR_{t,\cE}(\xi_1,\xi_2) &= \sup\{t (f(\xi_1) - f(\xi_2)) - \cE(f) \mid f \in \cF(X)\}\\
&\leq \sup\{t (f(\xi_1) - f(\xi_2)) -  \frac{1}{4\varepsilon} |f(\xi_1) - f(\xi_2)|^2 \mid f \in \cF(X)\} \\
&= \frac{1}{2\varepsilon} \frac{(2\varepsilon t)^2}{2} = \varepsilon t^2.
\end{align*}
Combining both inequalities yields the upper bound for $\cR_{t,\cE}$.

For the converse inequality, we note that by Lemma~\ref{lemma:slightly different formula resistance}  we have $\cE_i(f_i - f(\xi_i) + \alpha_i) = \cE_i(f_i)$ for $\alpha_i \in \R$ and $f_i \in \cF(X_i)$, $i =1,2$. Hence, testing with the function $f + (\alpha_1 - f(\xi_1))1_{X_1} + (\alpha_2 - f(\xi_2))1_{X_2}$, with $f|_{X_i} \in D(\cE_i)$ and $\alpha_i \in \R$, in the definition of $\cR_{t,\cE}(x_1,x_2)$, yields
\begin{align*}
\cR_{t,\cE}(x_1,x_2) &\geq t (f(x_1) - f(\xi_1)) +   t(\alpha_1 - \alpha_2) + t(f(\xi_2) - f(x_2))\\
&- \cE_1(f|_{X_1} - f(\xi_1) + \alpha_1) - \cE_2(f|_{X_2}- f(\xi_2) + \alpha_2)  - \frac{1}{4\varepsilon} |\alpha_1 - \alpha_2|^2\\
 &= t (f(x_1) - f(\xi_1)) +   t(\alpha_1 - \alpha_2) + t(f(\xi_2) - f(x_2))\\
&- \cE_1(f|_{X_1}) - \cE_2(f|_{X_2})  - \frac{1}{4\varepsilon} |\alpha_1 - \alpha_2|^2.
\end{align*}
Since $f|_{X_1}$, $f|_{X_2}$ and $\alpha_1,\alpha_2$ can be chosen independently, taking the supremum over them yields the claim (again we use Example~\ref{example:basic duality}).  
\end{proof}

\section{Examples}\label{section:examples}

\subsection{Nonlinear graph resistance forms}

In this subsection we assume that $X$ is at most countable. Moreover, we assume that $w = \{w_{xy} \mid x,y \in X\}$ is a family of symmetric lower semicontinuous convex functions $w_{xy} \colon \R \to [0,\infty]$ with $w_{xy}(0) = 0$. We define the  functional 
$$\cE_w \colon \cF(X) \to [0,\infty], \quad \cE_w(f) = \sum_{x,y \in X} w_{xy}(f(x) - f(y)).$$
For convenience we assume $w_{xy} = w_{yx}$, because for  $\tilde w = \{\frac{1}{2} (w_{xy} + w_{yx}) \mid x,y \in X\}$ we have $\cE_w = \cE_{\tilde w}$.

\begin{example}[$p$-energy on weighted graphs] \label{example:p energy}
 A {\em weighted graph on $X$} is  a symmetric function $b \colon X \times X \to [0,\infty)$. It can be interpreted as an edge weight of an induced discrete graph $G_b = (X,E_b)$, where $(x,y) \in E_b$ if and only if $b(x,y) > 0$.  If we fix $1 \leq p < \infty$, then $w_{xy}(t) = \frac{1}{p} b(x,y) |t|^p$ is a family of convex functions as above and 
 $$\cE_w(f) = \sum_{x,y \in X}  \frac{1}{p} b(x,y)|f(x) - f(y)|^p$$
 is the well studied $p$-energy of the weighted graph $(X,b)$. Instead of constant $p$,  in the definition of $w$ it would also be possible to consider a symmetric function $p \colon X \times X \to [1,\infty)$. 
\end{example}

\begin{proposition}\label{proposition:graph resistance forms}
  $\cE_w$ is convex, symmetric and  lower semicontinuous. Moreover, all normal contractions operate on $\cE_w$. 
\end{proposition}
 \begin{proof}
  Convexity and symmetry are obvious, lower semicontinuity follows from the lower semicontinuity of $w_{xy}$, $x,y \in X$, and Fatou's lemma.
  
  For $x,y \in X$ and $g \in \cF(X)$ we use the notation $\nabla_{xy} g = g(x) - g(y)$. Given a normal contraction  $C \colon \R \to \R$, for $x,y \in X$ we find $|\lambda_{xy}| \leq 1$ with 
  $$\lambda_{xy} \nabla_{xy} g =  \nabla_{xy}  Cg.$$
  Hence, using Lemma~\ref{lemma:convexity estimate}, we obtain 
  \begin{align*}
   &w_{xy} (\nabla_{xy}f  + \nabla_{xy}Cg ) + w_{xy} (\nabla_{xy}f  - \nabla_{xy}Cg)\\
   &= w_{xy} (\nabla_{xy}f  + \lambda_{xy} \nabla_{xy}g ) + w_{xy} (\nabla_{xy}f  - \lambda_{xy} \nabla_{xy}g)\\
   &\leq w_{xy} (\nabla_{xy}f  +  \nabla_{xy}g ) + w_{xy} (\nabla_{xy}f  -\nabla_{xy}g).
  \end{align*}
 Summing up these inequalities yields $\cE_w(f + Cg) + \cE_w(f-Cg) \leq \cE_w(f + g) + \cE_w(f-g)$. 
 \end{proof}

In what follows we use the concepts of Subsection~\ref{subsection:convex functionals}   for the functions in $w$ on $\R$, which is naturally identified with its dual space.   

For the finiteness of the elementary resistance we associate a  graph $G_w$ with vertex set $X$ to the family $w$. More precisely, we say that $x,y \in X$ are  neighbors in $G_w$ and write $x \sim_w y$ if $w_{xy} \neq 0$. Since we assumed the symmetry condition on $w$, the graph is symmetric.  Moreover, it is readily verified that $x \sim_w y$ implies the existence of $t > 0$ such that  $(w_{xy})^*(t) < \infty$ (use that  $w_{xy}^*(t) = \infty$ for all $t > 0$ and $w_{xy}(0) = 0$ imply $w_{xy} = (w_{xy})^{**} = 0$).

\begin{proposition}\label{proposition:finiteness of resistance graph forms}
 \begin{enumerate}[(a)]
  \item   $\cE_w$ is a nonlinear resistance form if and only if $G_w$ is connected.  
  \item If for each $x \in X$ there exists $t > 0$ with $\sum_{y \in X} w_{xy}(t) < \infty$, then  $\cF_c(X) \subset M(\cE_w)$ and $D(\cE_w)$ separates the points of $X$. 
 \end{enumerate}
\end{proposition}
\begin{proof}
 (a): All properties except the finiteness of the elementary resistance have been discussed above.

 Assume that $G_w$ is connected. Let $x,y \in X$ and assume that $\cE_w(f) \leq 1$. If $x \sim_w y$, then $w_{xy} \neq 0$. Hence, there exists $t > 0$ with $(w_{xy})^*(t) < \infty$. As discussed in Proposition~\ref{proposition:fundamental inequalities}, we have 
 $$t (f(x) - f(y)) \leq \av{t}_{O,(w_{xy})^*}\av{f(x) - f(y)}_{L,w_{xy}}$$
and  $\av{t}_{O,(w_{xy})^*} < \infty$. The assumption $\cE_w(f) \leq 1$ leads to $w_{xy}(f(x) - f(y)) \leq 1$, which implies $\av{f(x) - f(y)}_{L,w_{xy}} \leq 1$. Using also the symmetry of $\cE_w$, for $0 \leq C_{xy} = \av{t}_{O,(w_{xy})^*}/t < \infty$ we find
$$|f(x) - f(y)| \leq C_{xy}.$$

If $x$ and $y$ are not neighbors in $G_w$, then we choose a path $x = x_0 \sim_w x_1 \sim_w \ldots \sim_w x_n = y$ and  constants $0 \leq C_{x_ix_{i-1}} < \infty$ as above, to obtain 
$$|f(y) - f(x)| \leq \sum_{i = 1}^n |f(x_i) - f(x_{i-1})| \leq \sum_{i = 1}^n C_{x_ix_{i-1}}. $$
Since $f \in \cF(X)$ with $\cE_w(f) \leq 1$ was arbitrary and the constants are independent of $f$, this shows $R(x,y)  < \infty$. 

Now assume that $G_w$ is not connected and let $\emptyset \neq W \neq X$ be a connected component of $G_w$. If $w_{xy} \neq 0$, then either $x,y \in W$ or $x,y \in X \setminus W$. In other words, $w_{xy} \neq 0$ implies $1_W(x) = 1_W(y)$ and for   $\alpha > 0$ we obtain $\cE_w( \alpha 1_W) = 0$. For $x \in W$, $y \in X \setminus W$ and $\alpha > 0$, this yields
$$R(x,y) \geq |\alpha 1_W(x) - \alpha 1_W(y)| = \alpha, $$
showing $R(x,y) = \infty$. 

(b): Using the symmetry of $w_{xy}(t)$ (in $x,y \in X$ and $t \in \R$), for $z \in X$ and $t > 0$ we obtain
$$\cE_w(t 1_{\{z\}}) = \sum_{x,y \in X} w_{xy}(t(1_{\{z\}}(x) - 1_{\{z\}}(y))) =  2 \sum_{y \in X} w_{zy}(t).$$
Hence, the assumption in (b) implies $t 1_{\{z\}} \in D(\cE_w)$ for some $t > 0$, i.e., $1_{\{z\}} \in M(\cE_w)$.
\end{proof}

\begin{remark}
 Non-quadratic energies of this type (under some stronger conditions on $w$) have been studied before, see e.g. \cite{Soa93,Soa93a,Soa94,LS96,Kas16,Kas20} and references therein. The idea to use Luxemburg seminorms in this context can also be found in these references.  However, their focus is quite different and we are not aware of a good definition of the resistance metric let alone its additivity over serial circuits in this case. Also the relation of lower semicontinuity and completeness of the modular space is not discussed there.  
 
 The most important example is certainly the $p$-energy of a weighted graph (Example~\ref{example:p energy}), which is closely related to the discrete $p$-Laplacian.
\end{remark}


\subsection{Hypergraph forms}

As in the previous section, we assume that $X$ is at most countable. We let  $\mathcal P_c(X) =  \{K \subset X \mid K \text{ finite}\}$  and assume that we are given $\mu  \colon \mathcal P_c(X) \to [0,\infty)$. We define the corresponding {\em hypergraph form} by 
$$\cE_\mu \colon \cF(X) \to [0,\infty], \quad \cE_\mu(f) = \sum_{K \in \mathcal P_c(X)} \mu(K) (\max_{x \in K} f(x) - \min_{x \in K} f(x))^2.$$
With essentially the same arguments as in the proof of Proposition~\ref{proposition:graph resistance forms}, it is readily verified that $\cE_\mu$ is a lower semicontinuous $2$-homogeneous convex functional and that all normal contractions operate on $\cE_\mu$.  

Similar to the previous subsection, we introduce a graph $G_\mu$ with vertex set $X$ to study finiteness of the elementary resistance. We say that $x,y \in X$ with $x \neq y$ are neighbors in $G_\mu$ and write $x \sim_\mu y$ if there exists $K \in \mathcal P_c(X)$ with $\mu(K) > 0$ and $x,y \in K$. 

\begin{proposition} \label{proposition:hypergraph form}
\begin{enumerate}[(a)]
 \item $\cE_\mu$ is a resistance form if and only if $G_\mu$ is connected.
 \item If for each $x \in X$ we have
 $$\sum_{W \in \mathcal P_c(X),\, x \in W} \mu(W) < \infty,$$
 then $\cF_c(X) \subset D(\cE_\mu) = M(\cE_\mu)$.
\end{enumerate}
 
\end{proposition}

\begin{proof}
(a): It suffices to study the finiteness of the elementary resistance. 

Assume that $G_\mu$ is connected.  If $x \sim_\mu y$, then there exists $K \in \mathcal P_c(X)$ such that $x,y \in K$. Then
$$|f(x) - f(y)|^2 \leq (\max_{x \in K} f(x) - \min_{x \in K} f(x))^2$$
and, using $\av{\cdot}_L = \cE_\mu^{1/2}$, we infer 
$$|f(x) - f(y)| \leq \frac{1}{\sqrt{\mu(K)}}\av{f}_L. $$
If $x$ and $y$ are not neighbors, we can chain this inequality along a finite path to obtain the finiteness of $R$. 

As in the proof of Proposition~\ref{proposition:finiteness of resistance graph forms} we can show $R(x,y) = \infty$ if $x$ and $y$ belong to different connected components of $G_\mu$. Hence, $\cE_\mu$ is not a resistance form if $G_\mu$ is not connected. 

(b): This follows directly from 
$$\cE_\mu(1_{\{x\}}) = \sum_{W \in \mathcal P_c(X), x \in W} \mu(W).$$
\end{proof}

\begin{remark}
 The hypergraph form is the convex functional whose subgradient equals the multi-valued hypergraph Laplacian (when considered on the Hilbert space $\ell^2(X)$ instead of the whole space $\cF(X)$). Both objects have recently been studied in the context of discrete curvature and nonlinear heat flows, see e.g. \cite{HG1,HG2,HG3,HG4,HG5} and references therein.  
\end{remark}

%
%

%

\subsection{$p$-resistance forms} \label{subsection:p-resistance forms}

In \cite{KS25} so-called $p$-resistance forms are introduced. We show that these are nonlinear resistance forms in the sense of our definition. 

We start with repeating the definition of $p$-resistance forms in our slightly different notation.  For $1 \leq q \leq \infty$ we let $|\cdot|_q$ be the $\ell^q$-norm on $\R^d$, i.e., 
$|x|_q =  (\sum_{i = 1}^d |x_i|^q)^{1/q}$ if $q < \infty$ and $|x|_\infty = \max_i |x_i|$. Moreover, given any functional $\Phi \colon \cF(X) \to (-\infty,\infty]$ and $f = (f_1,\ldots,f_d) \in D(\Phi)^d$  we simply write $\Phi(f)$ for the vector $\Phi(f) = (\Phi(f_1),\ldots,\Phi(f_d)) \in \R^d.$

\begin{definition}[Generalized $p$-contraction property]
We say that a functional $\cE \colon \cF(X) \to [0,\infty]$ satisfies the {\em generalized $p$-contraction property}, if for all $1 \leq q_1 \leq p \leq q_2 \leq \infty$, all $d_1,d_2 \in \N$, all $T \colon \R^{d_1} \to \R^{d_2}$ with $T(0) = 0$ and $|T(x)-T(y)|_{q_2} \leq |x-y|_{q_1}$, $x,y \in \R^{d_1}$, and all $f \in D(\cE)^{d_1}$ we have $T(f) \in D(\cE)^{d_2}$ and 
$$|\cE^{1/p}(T(f))|_{q_2} \leq |\cE^{1/p}(f)|_{q_1}. $$
\end{definition}

\begin{definition}[$p$-resistance form according to \cite{KS25}] \label{definition:p resistance form}
 Let $1 < p < \infty$. A functional $\cE \colon \cF(X) \to [0,\infty]$ is called a {\em $p$-resistance form} if the following are satisfied. 
 \begin{enumerate}[(RF1)$_p$]
  \item   $\ker \cE = \R \cdot 1$  and $\cE^{1/p}$ is a seminorm on the vector space $D(\cE)$.
  \item The quotient space $(D(\cE)/ \R \cdot 1, \cE^{1/p})$ is a Banach space. 
  \item For all $x,y \in X$ with $x \neq y$ there exists $f \in D(\cE)$ with $f(x) \neq f(y)$. 
  \item For all $x,y \in X$ the elementary resistance of $\cE$ satisfies $R(x,y) < \infty$. 
  \item $\cE$ satisfies the generalized $p$-contraction property.
 \end{enumerate}
\end{definition}

\begin{remark}
 The quantity that is denoted by $R_\cE$ in \cite{KS25} equals $R^p$ in our notation. Hence, the formulation of (RF4)$_p$ in \cite{KS25} is slightly different but equivalent to the above. 
\end{remark}

\begin{lemma}\label{lemma:contraction property for p resistance}
 If $\cE$ satisfies the generalized $p$-contraction property, then all normal contractions  operate on $\cE$. 
\end{lemma}
\begin{proof}
 Let $C \colon \R \to \R$ be a normal contraction. We want to apply the generalized $p$-contraction property for the parameters $d_1 = d_2 = 2$ and $q_1 = q_2 = p$ and the map 
 $$T \colon \R^2 \to \R^2, \quad T(w,z) = (\frac{w+z}{2} + C(\frac{w-z}{2}),\frac{w+z}{2} - C(\frac{w-z}{2})).  $$
 Clearly,  $T(0) =  0$ and so we need to verify $|T(x) - T(y)|_p \leq |x-y|_p$ for all $x,y \in \R^2$. 
 
Since $C$ is a normal contraction, for all $t,s \in \R$ there exists $\lambda \in  \R$ with $|\lambda| \leq 1$ such that $C(t) - C(s) = \lambda (t-s)$. Hence, for $x = (x_1,x_2),y=(y_1,y_2) \in \R^2$,  Lemma~\ref{lemma:convexity estimate} implies
 \begin{align*}
  |T(x) - T(y)|_p^p &= \left|\frac{x_1+x_2}{2} - \frac{y_1+y_2}{2} + \left(C(\frac{x_1-x_2}{2}) - C(\frac{y_1-y_2}{2})\right)\right|^p\\
  & \,\, \,+ \left|\frac{x_1+x_2}{2} - \frac{y_1+y_2}{2} - \left(C(\frac{x_1-x_2}{2})  - C(\frac{y_1-y_2}{2})\right)\right|^p\\
  &\leq  \left|\frac{x_1+x_2}{2} - \frac{y_1+y_2}{2} + \left(\frac{x_1-x_2}{2} - \frac{y_1-y_2}{2}\right)\right|^p\\
  & \,\, \,+ \left|\frac{x_1+x_2}{2} - \frac{y_1+y_2}{2} - \left(\frac{x_1-x_2}{2}  - \frac{y_1-y_2}{2}\right)\right|^p\\
  &= |x-y|_p^p.
 \end{align*}
 For $f,g \in \cF(X)$ with $f \pm g \in D(\cE)$, the generalized $p$-contraction property implies (after taking $p$-th powers on both sides of the inequality) 
 \begin{align*}
  \cE(f + Cg) + \cE(f - Cg) &= \cE(T_1(f + g,f-g)) + \cE(T_2(f + g,f-g))\\
  &= |\cE^{1/p}(T(f+g,f-g))|_p^p\\
  &\leq |\cE^{1/p}(f+g,f-g)|_p^p\\
  &= \cE(f+g) + \cE(f-g). \hfill \qedhere
 \end{align*}
\end{proof}

\begin{proposition}\label{proposition:p resistance}
For $1 < p < \infty$ any $p$-resistance form in the sense of \cite{KS25} is a nonlinear resistance form.  
\end{proposition}
\begin{proof}
 (nRF1): This follows from $\cE^{1/p}$ being a seminorm (RF1)$_p$. 
 
 (nRF3): This is assumption (RF4)$_p$.
 
 (nRF4): This was proven in Lemma~\ref{lemma:contraction property for p resistance}.
 
 (nRF2): Since $\cE^{1/p}$ is a seminorm, $\cE$ is symmetric and positively $p$-homogeneous. Hence, the Luxemburg seminorm of $\cE$ equals $\cE^{1/p}$. According to \cite[Proposition~6.4]{KS25} the Banach space  $(D(\cE)/ \R \cdot 1, \cE^{1/p})$ is reflexive (this follows from Clarkson-type inequalities, which are a consequence of the generalized $p$-contraction property, see  \cite[Proposition 2.3]{KS25}). Hence, in our terminology $\cE$ is reflexive and the lower semicontinuity of $\cE$ follows from Theorem~\ref{theorem:charcterization lower semicontinuity constant kernel}.
\end{proof}

\begin{remark}
\begin{enumerate}[(a)]
 \item Under the summability condition of Proposition~\ref{proposition:hypergraph form}, it can be verified that the hypergraph form $\cE_\mu$ discussed in the previous subsection is a $2$-resistance form.
 \item Since $p$-resistance forms are nonlinear resistance forms, we refer to \cite{KS25} for more examples. 
\end{enumerate}

\end{remark}

\bibliographystyle{plain}
 
\bibliography{literatur}

\end{document}